\documentclass[11pt,twoside,english]{article}
\usepackage{amsmath}
\usepackage{amsfonts}
\usepackage{mathrsfs}
\usepackage{color}
\usepackage{ amsmath, amsfonts, amssymb, amsthm, amscd}
\usepackage[T1]{fontenc}
\usepackage[english]{babel}
\usepackage[noinfoline]{imsart}

\newtheorem{theorem}{Theorem}[section]
\newtheorem{lemma}[theorem]{Lemma}
\newtheorem{proposition}[theorem]{Proposition}

\newtheorem{definition}[theorem]{Definition}

\newtheorem{remark}{Remark}[section]

\newcommand{\cE}{\ensuremath{\mathcal E}}

\newcommand{\cH}{\ensuremath{\mathcal H}}

\newcommand{\cO}{\ensuremath{\mathcal O}}
\newcommand{\cP}{\ensuremath{\mathcal P}}

\newcommand{\cR}{\ensuremath{\mathcal R}}

\newcommand{\cW}{\ensuremath{\mathcal W}}

\newcommand{\bbN}{{\ensuremath{\mathbb N}} }

\newcommand{\bbR}{{\ensuremath{\mathbb R}} }

\newcommand{\N}{\mathbb{N}}

\newcommand{\be}{\begin{equation}}
\newcommand{\ee}{\end{equation}}
\newcommand{\beq}{\begin{eqnarray}}
\newcommand{\eeq}{\end{eqnarray}}

\newcommand{\1}{{1} \hspace{-0.25 em}{\rm I}}

\newcommand{\R}{\mathbb{R}}

\newcommand{\ced}{\end{proof}}

\begin{document}
\begin{frontmatter}
\title{The existence and uniqueness result for Quasilinear Stochastic PDEs with Obstacle under  weaker integrability conditions}
\date{April 16th 2013}
\runtitle{}
\author{\fnms{Laurent}
 \snm{DENIS}\corref{}\ead[label=e1]{ldenis@univ-evry.fr}}
\thankstext{T1}{The work of the first and third author is supported by the chair \textit{risque de cr\'edit}, F\'ed\'eration bancaire Fran\c{c}aise}
\address{Universit\'e
d'Evry-Val-d'Essonne-FRANCE
\\\printead{e1}}
\author{\fnms{Anis}
 \snm{MATOUSSI}\corref{}\ead[label=e2]{anis.matoussi@univ-lemans.fr}}
\thankstext{t2}{The research of the second author was partially supported by the Chair {\it Financial Risks} of the {\it Risk Foundation} sponsored by Soci\'et\'e G\'en\'erale, the Chair {\it Derivatives of the Future} sponsored by the {F\'ed\'eration Bancaire Fran\c{c}aise}, and the Chair {\it Finance and Sustainable Development} sponsored by EDF and Calyon }
\address{
 LUNAM Université, Université du Maine - FRANCE \\\printead{e2}}

\author{\fnms{Jing}
 \snm{ZHANG}\corref{}\ead[label=e3]{jzhang@univ-evry.fr}}
\address{Universit\'e
d'Evry-Val-d'Essonne -FRANCE\\\printead{e3}}

\runauthor{L. Denis, A. Matoussi and J. Zhang}

\begin{abstract}
We prove an existence and uniqueness result for quasilinear Stochastic PDEs with Obstacle (in short OSPDE) under a weaker integrability condition 
on the coefficient and the barrier.

\end{abstract}


\end{frontmatter}

\section{Introduction}
In this paper, we consider an obstacle problem for the following parabolic Stochastic PDE (SPDE in short)
\begin{equation}\label{SPDEO}\left\{ \begin{split}&du_t(x)=\partial_i  \left(a_{i,j}(x)\partial_ju_t(x)+g_i(t,x,u_t(x),\nabla
u_t(x))\right)dt+f(t,x,u_t(x),\nabla u_t(x))dt \\&\quad \quad\ \  \
\ \ \ +\sum_{j=1}^{+\infty}h_j(t,x,u_t(x),\nabla
u_t(x))dB^j_t +\nu (t,dx), \\
&u_t\geq S_t \, , \ \ \\ &u_0=\xi\, .\
\end{split}\right.\end{equation}
Here $a$ is a matrix defining a symmetric operator on an open domain $\cO$, with null boundary condition, $f,g,h$ are random coefficients and  $S$ is the given obstacle which is dominated by a given solution, $S'$, of a quasi-linear parabolic SPDE without obstacle.\\
In a recent work \cite{DMZ12} we have proved existence and uniqueness of the solution to this equation (\ref{SPDEO}) under standard Lipschitz hypotheses  and $L^2$-type integrability conditions on the coefficients. Let us recall that the solution is  a couple $(u,\nu )$, where $u$ is a process with values in the first order Sobolev space and $\nu$ is a random regular measure forcing $u$ to stay above $S$ and satisfying a minimal Skohorod condition.

 The study of the $L^p-$norms w.r.t. the randomness of the space-time uniform norm on the trajectories of a stochastic PDE was started by N. V. Krylov in \cite{Krylov}, for a more complete overview of existing works on this subject see \cite{DMS09,DM11} and the references therein. Concerning the obstacle problem, there are two approaches, a probabilistic one (see \cite{MatoussiStoica, Klimsiak}) based on the Feynmann-Kac's formula via the backward doubly stochastic differential equations and the analytical one (see \cite{DonatiPardoux,NualartPardoux,XuZhang}) based on the Green function.
\\

 In order to give a rigorous meaning to the notion of solution to this obstacle problem and inspired by the works of M. Pierre in the deterministic case (see \cite{Pierre,PIERRE}), we introduce the notion of parabolic capacity. The key point is that in \cite{DMZ12}, we construct a solution which admits a quasi continuous version hence defined outside a polar set and that regular measures which in general are not absolutely continuous w.r.t. the Lebesgue measure, do not charge polar sets. Moreover, in this reference, we have established an It\^o formula which is an essential tool in the present paper in order to get the comparison theorem.\\
 
 The aim of this paper is to relax the integrability assumptions both on $f$ and $S'$.  More precisely, we prove existence and uniqueness if the random coefficient
$f^0 (t,x) := f(t,x,0,0)$ belongs to a  class  $L^{p,q} (\cO \times [0,T])$  of functions $L^p$-integrable in space and $L^q$-integrable in time for $(p,q)$ such that $(1/p ,1/q)$ belongs to a certain interval with end points $p=2, q=1$ and $p=\frac{2^*}{2^* -1}, q=2$ respectively. Here, $2^* >2$ denotes a positive constant which depends on the dimension of the space.\\

The paper is organized as follows: in section 2 we introduce notations, hypotheses and the basic definitions related to the parabolic potential theory. In section 3, we prove an existence and uniqueness result for the
obstacle problem \eqref{SPDEO} with null Dirichlet condition under a weaker integrability hypothesis on $f$ and $S'$ and also give an estimate of
the positive part of the solution. Thanks to this estimate, we establish a comparison theorem for the solutions in section 4.
The last section is an Appendix in which we give  the proofs of several lemmas.

\section{Preliminaries}
\subsection{$L^{p,q}-$space}
Let $\mathcal{O}\subset \bbR^d$ be an open domain  and
$L^2(\mathcal{O})$ the set of square integrable functions with
respect to the Lebesgue measure on $\mathcal{O}$, it is an Hilbert
space equipped with the usual scalar product and norm as follows
$$(u,v)=\int_\mathcal{O}u(x)v(x)dx,\qquad\parallel u\parallel=(\int_\mathcal{O}u^2(x)dx)^{1/2}.$$
In general, we shall extend the notation
$$(u,v)=\int_\mathcal{O}u(x)v(x)dx,$$ where $u,\ v$ are measurable
functions defined on $\cO$ such that $uv\in L^1(\cO)$.
\\The first order Sobolev space of functions vanishing at the boundary will be denoted by $H_0^1(\cO)$,
its natural scalar product and norm are $$
\left( u,v\right) _{H_0^1\left( {\cal O}\right) }=\left( u,v\right) +\int_{%
{\cal O}}\sum_{i=1}^d\left( \partial _iu\left( x\right) \right)
\left(
\partial _iv\left( x\right) \right) dx,\;\left\| u\right\| _{H_0^1\left(
{\cal O}\right) }=\left( \left\| u\right\| _2^2+\left\| \nabla
u\right\| _2^2\right) ^{\frac 12}. $$As usual we shall denote
$H^{-1}(\cO)$ its dual space. \\
For each $t>0$ and for all real numbers $p,\,q\geq 1$, we denote by $%
L^{p,q}([0,t]\times {\cal O})$ the space of (classes of) measurable
functions $u:[0,t]\times {\cal O}\longrightarrow \mathbb{{R}}$ such
that
$$
\Vert u\Vert _{p,q;\,t}:=\left( \int_0^t\left( \int_{{\cal O}%
}|u(s,x)|^p\,dx\right) ^{q/p}\,ds\right) ^{1/q} $$ is finite. The
limiting cases with $p$ or $q$ taking the value $\infty $ are also
considered with the use of the essential sup norm.
\\Now we introduce some other spaces of functions and discuss a certain duality between them. Like in \cite{DMS05} and \cite{DMS09},
for self-containeness, we recall the following definitions:
\\Let $(p_1,q_1)$, $(p_2,q_2)$ $\in[1,\infty]^2$ be fixed and set $$ I=I\left( p_1,q_1,p_2,q_2\right) :=\left\{ \left( p,q\right)
\in \left[ 1,\infty \right] ^2/\;\exists \;\rho \in \left[
0,1\right] s.t.\right. $$ $$ \left. \frac 1p=\rho \frac
1{p_1}+\left( 1-\rho \right) \frac 1{p_2},\frac 1q=\rho \frac
1{q_1}+\left( 1-\rho \right) \frac 1{q_2}\right\} . $$ This means
that the set of inverse pairs $\left( \frac 1p,\frac 1q\right) ,$
$(p,q)$
 belonging to $I,$ is a segment contained in the square $%
\left[ 0,1\right] ^2,$ with the extremities $\left( \frac
1{p_1},\frac 1{q_1}\right) $ and $\left( \frac 1{p_2},\frac
1{q_2}\right) .$ \\We introduce: $$ L_{I;t}=\bigcap_{\left(
p,q\right) \in I}L^{p,q}\left( \left[ 0,t\right] \times {\cal
O}\right) . $$ We know that this space coincides with the
intersection of the extreme spaces,
$$ L_{I;t}=L^{p_1,q_1}\left( \left[ 0,t\right] \times {\cal
O}\right) \cap L^{p_2,q_2}\left( \left[ 0,t\right] \times {\cal
O}\right) $$ and that it is a Banach space with the following norm$$
\left\| u\right\| _{I;t}:=\left\| u\right\| _{p_1,q_1;t}\vee \left\|
u\right\| _{p_2,q_2;t}. $$
The other space of interest is the algebraic sum%
$$ L^{I;t}:=\sum_{\left( p,q\right) \in I}L^{p,q}\left( \left[
0,t\right] \times {\cal O}\right) , $$ which represents the vector
space generated by the same family of spaces.
This is a normed vector space with the norm%
$$ \left\| u\right\|^{I;t} :=\,\inf \left\{ \sum_{i=1}^n\left\|
u_i\right\| _{p_i,q_i;\,t}\,/\;u=\sum_{i=1}^nu_i,u_i\in
L^{p_i,q_i}\left( \left[ 0,t\right] \times {\cal O}\right) ,\,
\left( p_i,q_i\right) \in I,\,i=1,...n;\,n\in
\mathbb{{N}}^{*}\right\} . $$
Clearly one has $L^{I;t}\subset L^{1,1}\left( \left[ 0,t\right] \times {\cal %
O}\right) $ and $\left\| u\right\| _{1,1;t}\le c\left\| u\right\|
^{I;t},$ for each $u\in L^{I;t},$ with a certain constant $c>0.$

We also remark that if $\left( p,q\right) \in I,$ then the conjugate pair $%
\left( p^{\prime },q^{\prime }\right) ,$ with $\frac 1p+\frac
1{p^{\prime }}=\frac 1q+\frac 1{q^{\prime }}=1,$ belongs to another
set, $I^{\prime },$
of the same type. This set may be described by%
$$ I^{\prime }=I^{\prime }\left( p_1,q_1,p_2,q_2\right) :=\left\{
\left( p^{\prime },q^{\prime }\right) /\;\exists \left( p,q\right)
\in I\;s.t.\;\frac 1p+\frac 1{p^{\prime }}=\frac 1q+\frac
1{q^{\prime }}=1\right\} $$ and it is not difficult to check that
$I^{\prime }\left( p_1,q_1,p_2,q_2\right) =I\left( p_1^{\prime
},q_1^{\prime },p_2^{\prime },q_2^{\prime }\right) ,$ where
$p_1^{\prime },q_1^{\prime },p_2^{\prime }$ and $q_2^{\prime }$ are
defined by $\frac 1{p_1}+\frac 1{p_1^{\prime }}=\frac 1{q_1}+\frac
1{q_1^{\prime }}=\frac 1{p_2}+\frac 1{p_2^{\prime }}=\frac
1{q_2}+\frac 1{q_2^{\prime }}=1.$

Moreover, by H\"older's inequality, it follows that one has
\begin{equation}
\label{dual}\int_0^t\int_{{\cal O}}u\left( s,x\right) v\left(
s,x\right) dxds\le \left\| u\right\| _{I;t}\left\| v\right\|
^{I^{\prime };t},
\end{equation}
for any $u\in L_{I;t}$ and $v\in L^{I^{\prime };t}.$ This inequality
shows
that the scalar product of $L^2\left( \left[ 0,t\right] \times {\cal O}%
\right) $ extends to a duality relation for the spaces $L_{I;t}$ and
$L^{I^{\prime };t}.$

Now let us recall that the Sobolev inequality states that%
\begin{equation}\label{Sobolev}
\left\| u\right\| _{2^{*}}\le c_S\left\| \nabla u\right\| _2,
\end{equation}
for each $u\in H_0^1\left( {\cal O}\right) ,$ where $c_S>0$ is
a constant
that depends on the dimension and $2^{*}=\frac{2d}{d-2}$ if $d>2,$ while $%
2^{*}$ may be any number in $]2,\infty [$ if $d=2$ and $2^{*}=\infty $ if $%
d=1.$ Therefore one has%
$$ \left\| u\right\| _{2^{*},2;t}\le c_S\left\| \nabla u\right\|
_{2,2;t}, $$ for each $t\ge 0$ and each $u\in L_{loc}^2\left(
\mathbb{R}_{+};H_0^1\left( {\cal O}\right) \right) .$ If $u\in
L_{loc}^{\infty}\left( \mathbb{R}_{+}; L^2\left( {\cal O}\right) \,
\right) \bigcap L^2_{loc} \left( \mathbb{R}_+;  H_0^1\left( {\cal
O}\right) \right),$ one has
$$\left\| u\right\| _{2,\infty ;t}\vee \left\| u\right\|
_{2^{*},2;t}\le c_1\left( \left\| u\right\| _{2,\infty ;t}^2+\left\|
\nabla u\right\| _{2,2;t}^2\right) ^{\frac 12},$$with $c_1=c_S\vee
1.$

One particular case of interest for us in relation with this
inequality is when $p_1=2,q_1=\infty $ and $p_2=2^{*},q_2=2.$ If
$I=I\left( 2,\infty ,2^{*},2\right) ,$ then the corresponding set of
associated conjugate numbers is $I^{\prime }=I^{\prime }\left(
2,\infty ,2^{*},2\right) =I\left( 2,1,\frac{2^{*}}{2^{*}-1},2\right)
,$ where for $d=1$ we make the convention that
$\frac{2^{*}}{2^{*}-1}=1.$ In this particular case we shall use the
notation $L_{\#;t}:=L_{I;t}$ and $L_{\#;t}^*:=L^{I^{\prime };t}$ and
the
respective norms will be denoted by%
$$ \left\| u\right\| _{\#;t}:=\left\| u\right\| _{I;t}=\left\|
u\right\| _{2,\infty ;t}\vee \left\| u\right\|
_{2^{*},2;t},\;\left\| u\right\|_{\#;t}^*:=\left\| u\right\|
^{I^{\prime };t}. $$ Thus we may write
\begin{equation}\label{sobolev}\
\left\| u\right\| _{\#;t}\le c_1\left( \left\| u\right\| _{2,\infty
;t}^2+\left\| \nabla u\right\| _{2,2;t}^2\right) ^{\frac 12},
\end{equation}
for any  $u\in L_{loc}^{\infty}\left( \mathbb{{R}}_{+}; L^2\left(
{\cal O}\right) \, \right) \bigcap L^2_{loc} \left( \mathbb{R}_+;
H_0^1\left( {\cal O}\right) \right)$ and $t\ge 0$ and the
duality inequality becomes%
\begin{equation}\label{dual2}\int_0^t\int_{{\cal O}}u\left( s,x\right) v\left( s,x\right)
dxds\le \left\| u\right\| _{\#;t}\left\| v\right\|_{\#;t}^*,
\end{equation} for any $u\in L_{\#;t}$ and $v\in L_{\#;t}^*.$


\subsection{Hypotheses}
We consider a sequence $((B^i(t))_{t\geq0})_{i\in\mathbb{N}^*}$ of
independent Brownian motions defined on a standard filtered
probability space $(\Omega,\mathcal{F},(\mathcal{F}_t)_{t\geq0},P)$
satisfying the usual conditions.\\
Let $A$ be a symmetric second order differential operator defined on the open subset $ \cO \subset \R^d$, with
domain $\mathcal{D}(A)$, given by
$$A:=-L=-\sum_{i,j=1}^d\partial_i(a^{i,j}\partial_j).$$ We assume that
$a=(a^{i,j})_{i,j}$ is a measurable symmetric matrix defined on
$\mathcal{O}$ which satisfies the uniform ellipticity
condition$$\lambda|\xi|^2\leq\sum_{i,j=1}^d
a^{i,j}(x)\xi^i\xi^j\leq\Lambda|\xi|^2,\ \forall x\in\mathcal{O},\
\xi\in \R^d,$$where $\lambda$ and $\Lambda$ are positive constants.
The energy associated with the matrix $a$ will be denoted by
\begin{equation}
\label{energy}
 \mathcal{E} \left( w,v\right)=\sum_{i,j=1}^d
\int_{\cO}a^{i,j}(x)\partial_i w(x)\partial_j v(x)\, dx,\ \ \forall w,\, v \in  H^1_{0}
(\mathcal{O} ).
\end{equation}
We consider the quasilinear stochastic partial differential equation
\eqref{SPDEO} with initial condition $u(0,\cdot)=\xi(\cdot)$ and Dirichlet
boundary condition $u(t,x)=0,\ \forall\ (t,x)\in
\bbR^+\times\partial\mathcal{O}$.

We assume that we have predictable random
functions\begin{eqnarray*}&&f:\R_+\times\Omega\times\mathcal{O}\times
\R\times \R^d\rightarrow
\R,\\&&g=(g_1,...,g_d):\R_+\times\Omega\times\mathcal{O}\times \R\times
\R^d\rightarrow
\R^d,\\&&h=(h_1,...,h_i,...):\R_+\times\Omega\times\mathcal{O}\times
\R\times \R^d\rightarrow \R^{\mathbb{N}^*}.\end{eqnarray*} We define
\begin{equation*}
\begin{split}
 &f ( \cdot
,\cdot,\cdot, 0,0):=f^0, \  g( \cdot,\cdot,\cdot ,0,0) :=g^0 =
(g_1^0,...,g_d^0) \ \mbox{and}
\  h( \cdot,\cdot,\cdot ,0,0) :=h^0 = (h_1^0,...,h_i^0,...) .\\
\end{split}
\end{equation*}
In the sequel, $|\cdot|$ will always denote the underlying Euclidean
or $l^2-$norm. For
example$$|h(t,\omega,x,y,z)|^2=\sum_{i=1}^{+\infty}|h_i(t,\omega,x,y,z)|^2.$$
\begin{remark} Let us note that this general setting of the SPDE \eqref{SPDEO} we consider, encompasses the case of an SPDE driven by a space-time noise, colored in space and white in time  as in \cite{SW} for example (see also Example 1 in \cite{DMZ12}).

\end{remark}
\textbf{Assumption (H):} There exist non-negative constants $C,\
\alpha,\ \beta$ such that for almost all $\omega$, the following
inequalities hold for all
$(x,y,z,t)\in\mathcal{O}\times\mathbb{R}\times\mathbb{R}^d\times\mathbb{R}_+$:\begin{enumerate}
\item   $|f(t,\omega,x,y,z)-f(t,\omega,x,y',z')|\leq C(|y-y'|+|z-z'|),$
\item $|g(t,\omega,x,y,z)-g(t,\omega,x,y',z')|\leq
C|y-y'|+\alpha|z-z'|,$
\item $|h(t,\omega,x,y,z)-h(t,\omega,x,y',z')|\leq
C|y-y'|+\beta|z-z'|,$
\item the contraction property: $2\alpha+\beta^2<2\lambda.$
\end{enumerate}

Moreover we  introduce some integrability conditions  on the
coefficients $f^0, \;
g^0, \, h^0$ and the initial data $\xi$ : we fix a terminal time $T>0$. \\

 \textbf{Assumption (HI2)}   $$  E \left(\|\xi \|_2^2   +\left\| f^0\right\|_{2,2;t}^2+\left\|
\left|g^0\right|\right\| _{2,2;t}^2+\left\|\left| h^0\right|\right\| _{2,2;t}^2\right) <\infty
, $$ for each $t\in[0,T]$.
\\[0.2cm]
\textbf{Assumption (HI\#)}
$$ E\left(  \|\xi \|_2^2+\left( \left\| f^0\right\|_{\#;t}^*\right) ^2+\left\|
\left|g^0\right|\right\| _{2,2;t}^2+\left\| \left|h^0\right|\right\| _{2,2;t}^2\right) <\infty
, $$ for each $t\in[0,T].$
\\[0.2cm]

\begin{remark}\label{injL2}
Note that $\left( 2,1\right) $ is the pair of conjugates of the pair
$\left( 2,\infty \right) $ and so $\left( 2,1\right) $ belongs to
the set $I^{\prime }$ which defines the space $L_{\#;t}^*.$ Since
$\left\| v\right\| _{2,1;t}\le \sqrt{t}\left\| v\right\| _{2,2;t}$
for each $v\in L^{2,2}\left( \left[
0,t\right] \times {\cal O}\right) ,$ it follows that%
$$ L^{2,2}\left( \left[ 0,t\right] \times {\cal O}\right) \subset
L^{2,1;t}\subset L_{\#;t}^*, $$ and $\left\| v\right\|_{\#;t}^*\le
\sqrt{t}\left\| v\right\| _{2,2;t},$ for each $v\in L^{2,2}\left(
\left[ 0,t\right] \times {\cal O}\right) .$ This shows that the
condition \textbf{(HI\#)} is weaker than \textbf{(HI2)}.
\end{remark}
\subsection{Weak solutions}
We now introduce $\cH_T$, the space of
  $H_0^1(\mathcal{O})$-valued
  predictable processes $(u_t)_{t \in[0,T]}$ such that
\[
\left( E \sup_{0\leq s\leq T} \left\| u_{s}\right\|_2
^{2}+\int_{0}^{T}E\, \mathcal{E}\left( u_{s}\right) ds\right)
^{1/2}\;< \; \infty \;.
\]

The space of test functions is the algebraic tensor product $\mathcal{D}=\mathcal{C}
_{c}^{\infty }(\bbR^+)\otimes \mathcal{C}_c^2 (\cO )$, where $\mathcal{C}
_{c}^{\infty }(\bbR^+)$ denotes the space of all real infinite
differentiable  functions with compact support in $\mathbb{R}^+$ and
$\mathcal{C}_c^2 (\cO )$ the set of $C^2$-functions with compact
support in $\cO$.

\vspace{0.5cm}

Now we recall the definition of the regular measure which has been defined in \cite{DMZ12}.\\
$\mathcal{K}$ denotes $L^\infty([0,T];L^2(\mathcal{O}))\cap
L^2([0,T];H_0^1(\mathcal{O}))$ equipped with the norm:
\begin{eqnarray*}\parallel
v\parallel^2_\mathcal{K}&=&\parallel
v\parallel^2_{L^\infty([0,T];L^2(\mathcal{O}))}+\parallel
v\parallel^2_{L^2([0,T];H_0^1(\mathcal{O}))}\\
&=&\sup_{t\in[0,T[}\parallel v_t\parallel^2 +\int_0^T \left(
\parallel v_t \parallel^2  +\mathcal{E}(v_t)\right)\, dt
.\end{eqnarray*} $\mathcal{C}$ denotes the space of continuous
functions with compact support in $[0,T[\times\mathcal{O}$ and
finally:
$$\mathcal {W}=\{\varphi\in L^2([0,T];H_0^1(\mathcal{O}));\ \frac{\partial\varphi}{\partial t}\in
L^2([0,T];H^{-1}(\mathcal{O}))\}, $$ endowed with the
norm$\parallel\varphi\parallel^2_{\mathcal {W}}=\parallel
\varphi\parallel^2_{L^2([0,T];H_0^1(\mathcal{O}))}+\parallel\displaystyle\frac{\partial
\varphi}{\partial t}\parallel^2_{L^2([0,T];H^{-1}(\mathcal{O}))}$. \\
It is known (see \cite{LionsMagenes}) that $\mathcal{W}$ is
continuously embedded in $C([0,T]; L^2 (\cO))$, the set of $L^2 (\cO
)$-valued continuous functions on $[0,T]$. So without ambiguity, we
will also consider
$\mathcal{W}_T=\{\varphi\in\mathcal{W};\varphi(T)=0\}$,
$\mathcal{W}^+=\{\varphi\in\mathcal{W};\varphi\geq0\}$,
$\mathcal{W}_T^+=\mathcal{W}_T\cap\mathcal{W}^+$.
\begin{definition}
An element $v\in \mathcal{K}$ is said to be a {\bf parabolic
potential} if it satisfies:
$$ \forall\varphi\in\mathcal{W}_T^+,\
\int_0^T-(\frac{\partial\varphi_t}{\partial
t},v_t)dt+\int_0^T\mathcal{E}(\varphi_t,v_t)dt\geq0.$$ We denote by
$\mathcal{P}$ the set of all parabolic potentials.
\end{definition}
The next representation property is  crucial:
\begin{proposition}(Proposition 1.1 in \cite{PIERRE})\label{presentation}
Let $v\in\mathcal{P}$, then there exists a unique positive Radon
measure on $[0,T[\times\mathcal{O}$, denoted by $\nu^v$, such that:
$$\forall\varphi\in\mathcal{W}_T\cap\mathcal{C},\ \int_0^T(-\frac{\partial\varphi_t}{\partial t},v_t)dt+\int_0^T\mathcal{E}(\varphi_t,v_t)dt=\int_0^T\int_\mathcal{O}\varphi(t,x)d\nu^v.$$
Moreover, $v$ admits a right-continuous (resp. left-continuous)
version $\hat{v} \ (\makebox{resp. } \bar{v}): [0,T]\mapsto L^2
(\cO)$ .\\
Such a Radon measure, $\nu^v$ is called {\bf a regular measure} and
we write:
$$ \nu^v =\frac{\partial v}{\partial t}+Av .$$
\end{proposition}

\begin{definition}
Let $K\subset [0,T[\times\mathcal{O}$ be compact, $v\in\mathcal{P}$
is said to be  \textit{$\nu-$superior} than 1 on $K$, if there exists a
sequence $v_n\in\mathcal{P}$ with $v_n\geq1\ a.e.$ on a neighborhood
of $K$ converging to $v$ in $L^2([0,T];H_0^1(\mathcal{O}))$.
\end{definition}
We denote:$$\mathscr{S}_K=\{v\in\mathcal{P};\ v\ is\ \nu-superior\
to\ 1\ on\ K\}.$$
\begin{proposition}(Proposition 2.1 in \cite{PIERRE})
Let $K\subset [0,T[\times\mathcal{O}$ compact, then $\mathscr{S}_K$
admits a smallest $v_K\in\mathcal{P}$ and the measure $\nu^v_K$
whose support is in $K$ satisfies
$$\int_0^T\int_\mathcal{O}d\nu^v_K=\inf_{v\in\mathcal{P}}\{\int_0^T\int_\mathcal{O}d\nu^v;\ v\in\mathscr{S}_K\}.$$
\end{proposition}
\begin{definition}(Parabolic Capacity)\begin{itemize}
                                        \item Let $K\subset [0,T[\times\mathcal{O}$ be compact, we define
$cap(K)=\int_0^T\int_\mathcal{O}d\nu^v_K$;
                                        \item let $O\subset
[0,T[\times\mathcal{O}$ be open, we define $cap(O)=\sup\{cap(K);\
K\subset O\ compact\}$;
                                        \item   for any borelian
$E\subset [0,T[\times\mathcal{O}$, we define $cap(E)=\inf\{cap(O);\
O\supset E\ open\}$.
                                      \end{itemize}

\end{definition}
\begin{definition}A property is said to hold quasi-everywhere (in short q.e.)
if it holds outside a set of null capacity.
\end{definition}
\begin{definition}(Quasi-continuous)

\noindent A function $u:[0,T[\times\mathcal{O}\rightarrow\mathbb{R}$
 is called quasi-continuous, if there exists a decreasing sequence of open
subsets $O_n$ of $[0,T[\times\mathcal{O}$ with: \begin{enumerate}
                        \item for all $n$, the restriction of $u_n$ to the complement of $O_n$ is
continuous;
                                       \item $\lim_{n\rightarrow+\infty}cap\;(O_n)=0$.
                                     \end{enumerate}
We say that $u$ admits a quasi-continuous version, if there exists
$\tilde{u}$ quasi-continuous  such that $\tilde{u}=u\ a.e.$
\end{definition}
The next proposition, whose proof may be found in \cite{Pierre} or \cite{PIERRE} shall play an important role in the sequel:
\begin{proposition}\label{Versiont} Let $K\subset \cO$ a compact set, then $\forall t\in [0,T[$
$$cap (\{ t\}\times K)=\lambda_d (K),$$
where $\lambda_d$ is the Lebesgue measure on $\cO$.\\
As a consequence, if $u: [0,T[\times \cO\rightarrow \R$ is a map defined quasi-everywhere then it defines uniquely a map from $[0,T[$ into $L^2 (\cO)$.
In other words, for any $t\in [0,T[$, $u_t$ is defined without any ambiguity as an element in $L^2 (\cO)$.
Moreover, if $u\in \mathcal{P}$, it admits  version $\bar{u}$ which is left continuous on $[0,T]$ with values in $L^2 (\cO )$ so that $u_T =\bar{u}_{T^-}$ is also defined without ambiguity.
\end{proposition}
\begin{remark} The previous proposition applies if for example $u$ is quasi-continuous.
\end{remark}

We end this part by a convergence lemma which plays an important
role in our approach (Lemma 3.8 in \cite{PIERRE}):
\begin{lemma}\label{convergemeas}
If $v^n\in\mathcal{P}$ is a bounded sequence in $\mathcal{K}$ and
converges weakly to $v$ in $L^2([0,T];H_0^1(\mathcal{O}))$; if $u$ is
a quasi-continuous function and $|u|$ is bounded by a element in
$\mathcal{P}$. Then
$$\lim_{n\rightarrow+\infty}\int_0^T\int_\mathcal{O}ud\nu^{v^n}=\int_0^T\int_\mathcal{O}ud\nu^{v}.$$
\end{lemma}
\vspace{0.5cm}

We now give the assumptions on the obstacle that we shall need in the different cases that we shall consider.\\

\textbf{Assumption (O):} The obstacle $S: [0,T]\times \Omega\times \cO\rightarrow \R$ is  an
adapted random field  almost surely quasi-continuous, in the sense that for $P$-almost all $\omega\in\Omega$, the map $(t,x)\rightarrow S_t (\omega,x)$ is quasi-continuous.  Moreover,  $S_0 \leq \xi$ $P$-almost surely and $S$ is
controlled by the solution of an SPDE, i.e. $\forall t\in[0,T],$
\begin{equation}S_t\leq S'_t,\quad dP\otimes dt\otimes dx-a.e.\end{equation} where
$S'$ is the solution of the linear SPDE
\begin{equation}\left\{\begin{array}{ccl} \label{obstacle}
 dS'_t&=&LS'_tdt+f'_tdt+\sum_{i=1}^d \partial_i g'_{i,t}dt+\sum_{j=1}^{+\infty}h'_{j,t}dB^j_t\\
                  S'(0)&=&S'_0 ,
\end{array}\right. \end{equation}
with null boundary Dirichlet conditions.\\

\textbf{Assumption (HO2)}   $$  E \left(\|\xi \|_2^2   +\left\| f'\right\|_{2,2;T}^2+\left\|
|g'|\right\| _{2,2;T}^2+\left\| |h'|\right\| _{2,2;T}^2\right) <\infty .
$$
\textbf{Assumption (HO\#)}
$$ E\left(  \|S'_0 \|_2^2+\left( \left\| f'\right\|_{\#;T}^*\right) ^2+\left\|
|g'|\right\| _{2,2;T}^2+\left\| |h'|\right\| _{2,2;T}^2\right) <\infty .
$$
\begin{remark}\label{remark3} It is well-known that under {\bf (HO2)} $S'$ belongs to $\cH_T$, is unique and satisfies the following
estimate:
\begin{equation}\label{estimobstacle1}
E\sup_{t\in[0,T]}\parallel S'_t\parallel^2+E\int_0^T\mathcal{E}(S'_t)dt\leq CE\left[\parallel S'_0\parallel^2+\int_0^T(\parallel f'_t\parallel^2+\parallel |g'_t|\parallel^2+\parallel |h'_t|\parallel^2)dt\right],
\end{equation}
see for example Theorem 8 in \cite{DenisStoica}. Moreover, as a consequence of  Theorem 3 in \cite{DMZ12}, we know that $S'$ admits a
quasi-continuous version.\\
Under the weaker condition {\bf(HO\#)}, $S'$ also exists, is unique and satisfies the following estimate (see Theorem 3 in \cite{DMS09}):
\begin{equation}\label{estimobstacle2}
E\sup_{t\in[0,T]}\parallel S'_t\parallel^2+E\int_0^T\mathcal{E}(S'_t)dt\leq CE\left[\parallel S'_0\parallel^2+\int_0^T\left((\parallel f'_t\parallel_\#^*)^2+\parallel |g'_t|\parallel^2+\parallel |h'_t|\parallel^2\right)dt\right].
\end{equation}
\end{remark}
\begin{definition} A pair
$(u,\nu)$ is said to be a solution of the problem (\ref{SPDEO}) if
\begin{enumerate}
    \item $u\in\mathcal{H}_T$, $u(t,x)\geq S(t,x),\ dP\otimes dt\otimes
    dx-a.e.$ and $u_0(x)=\xi,\ dP\otimes dx-a.e.$;
    \item $\nu$ is a random regular measure defined on
    $[0,T[\times\mathcal{O}$;
    \item the following relation holds almost surely, for all
    $t\in[0,T]$ and all $\varphi\in\mathcal{D}$,
     \begin{equation}\begin{split}\label{solution}(u_t,\varphi_t)=&(\xi,\varphi_0)+\int_0^t(u_s,\partial_s\varphi_s)ds-\int_0^t\mathcal{E}(u_s,\varphi_s)ds\\&-\sum_{i=1}^d\int_0^t(g^i_s(u_s,\nabla u_s),\partial_i\varphi_s)ds
    +\int_0^t(f_s(u_s,\nabla u_s),\varphi_s)ds\\&+\sum_{j=1}^{+\infty}\int_0^t(h^j_s(u_s,\nabla u_s),\varphi_s)dB^j_s+\int_0^t\int_{\mathcal{O}}\varphi_s(x)\nu(dx,ds);\end{split}\end{equation}
    \item $u$ admits a quasi-continuous version, $\tilde{u}$, and we have  $$\int_0^T\int_\cO(\tilde{u}(s,x)-S(s,x))\nu(dx,ds)=0,\ \
    P-a.s.$$
  \end{enumerate}
\end{definition}
Finally, in the sequel, we introduce some constants $\epsilon$, $\delta>0$, we shall denote by $C_\epsilon$, $C_\delta$ some constants depending only on  $\epsilon$, $\delta$, typically those appearing in the kind of inequality
\begin{equation}\label{young}|ab|\leq \epsilon a^2 + C_\epsilon b^2.\end{equation}

\section{Main results}
In this section, we prove the existence and uniqueness under a weaker integrability on $f^0$ and $S'$, improving the results obtained in \cite{DMZ12}, Theorem 4 and then give an It\^o formula and estimate for the positive part of the solution, which is a crucial step leading to the comparison theorem.
Let us note that these results have been established in the case of SPDE without obstacle (see Section 3 in \cite{DMS09} and \cite{DM11}).
\subsection{Existence and uniqueness}\label{section3.1}
To get the estimates we need,  we apply It\^o's formula to $u-S'$, in order to take advantage of the fact that $S-S'$ is non-positive and that
as $u$ is solution of (\ref{SPDEO}) and $S'$ satisfies
(\ref{obstacle}), $u-S'$ satisfies
\begin{equation}\label{uminusS'}\left\{ \begin{split}&d(u_t-S'_t)=\partial_i  (a_{i,j}(x)\partial_j(u_t(x)-S'_t(x)))dt+(f(t,x,u_t(x),\nabla
u_t(x))-f'(t,x))dt \\&+\partial_i(g_i(t,x,u_t(x),\nabla
u_t(x))-g'_i(t,x))dt+(h_j(t,x,u_t(x),\nabla
u_t(x))-h'_j(t,x))dB^j_t\\&+\nu(x,dt), \\
 &(u-S')_0=\xi-S'_0\, ,\\&u-S'\geq S-S'\, . \end{split}\right.\end{equation}
that is why we introduce the following functions: \begin{eqnarray}\label{BAR}&&\bar{f}(t,\omega,x,y,z)=f(t,\omega,x,y+S'_t,z+\nabla
S'_t)-f'(t,\omega,x);\nonumber\\&&
\bar{g}(t,\omega,x,y,z)=g(t,\omega,x,y+S'_t,z+\nabla S'_t)-g'(t,\omega,x);\nonumber\\&&
\bar{h}(t,\omega,x,y,z)=h(t,\omega,x,y+S'_t,z+\nabla S'_t)-h'(t,\omega,x).\end{eqnarray}
Let us remark that the Skohorod condition for $u-S'$ is satisfied
since
$$ \int_0^T\int_{\cO} (u_s (x)-S'_s (x))-(S_s (x)-S'_s (x))\nu (ds ,dx)=\int_0^T\int_{\cO} (u_s (x)-S_s (x))\nu (ds ,dx)=0.$$
It is obvious that $\bar{f}$, $\bar{g}$ and $\bar{h}$ satisfy
the Lipschitz conditions with the same Lipschitz coefficients as $f$,
$g$ and $h$. Then, using Remark \ref{injL2}, we check the integrability conditions for
$\bar{f}^0$, $\bar{g}^0$ and $\bar{h}^0$:
\begin{eqnarray*}\left\|\bar{f}^0\right\|_{\#;T}^*&=&\left\|f(S',\nabla S')-f'\right\|_{\#;T}^*\leq\left\|f(S',\nabla S')\right\|_{\#;T}^*+\left\|f'\right\|_{\#;T}^*\\&\leq&\left\|f^0\right\|_{\#;T}^*+C\left\|S'\right\|_{\#;T}^*+C\left\|\nabla S'\right\|_{\#;T}^*+\left\|f'\right\|_{\#;T}^*\\&\leq&\left\|f^0\right\|_{\#;T}^*+C\sqrt{t}\left\|S'\right\|_{2,2;T}+C\sqrt{T}\left\|\nabla S'\right\|_{2,2;T}+\left\|f'\right\|_{\#;T}^* .\end{eqnarray*}
We know that (see Remark \ref{remark3}): $$E\left(\left\|S'\right\|^2_{2,2;T}+\left\|\nabla S'\right\|^2_{2,2;T}\right)<\infty .$$
Hence, for each $t$, we have
$$E\left(\left\|\bar{f}^0\right\|_{\#;T}^*\right)^2<\infty.$$
We also have: \begin{eqnarray*}\left\|\bar{g}^0\right\|
_{2,2;T}&=&\left\|\bar{g}(S',\nabla S')-g'\right\|
_{2,2;T}\leq\left\|\bar{g}(S',\nabla S')\right\|
_{2,2;T}+\left\|g'\right\| _{2,2;T}\\&\leq&\left\|g^0\right\|
_{2,2;T}+C\left\|S'\right\| _{2,2;T}+\alpha\left\|\nabla S'\right\|
_{2,2;T}+\left\|g'\right\| _{2,2;T}<\infty .\end{eqnarray*}
 And
the same thing for $\bar{h}$. Hence,
\begin{equation}\label{differenceweaker} E\left( \left( \left\| \bar{f}\right\|_{\#;T}^*\right) ^2+\left\|
\bar{g}\right\| _{2,2;T}^2+\left\| \bar{h}\right\|
_{2,2;T}^2\right) <\infty . \end{equation}

We now state the main Theorem of this subsection:
\begin{theorem}\label{2estimate} Under  conditions {\bf (H)}, {\bf (O)}, {\bf (HI\#)} and {\bf (HO\#)}, the obstacle problem (\ref{SPDEO}) admits
 a unique solution $(u,\nu)$, where $u$ is in $\cH_T$ and $\nu$ is a random regular measure.
We denote by $\cR^{\#}(\xi,f,g,h,S)$ the solution of OSPDE \eqref{SPDEO} when it exists and is unique.
\end{theorem}

For the proof of this theorem, we need the following two lemmas whose proofs are  given in the appendix. The first lemma concerns It\^o's formula for the solution of SPDE (\ref{SPDEO}) without obstacle under {\bf(H)} and {\bf(HI\#)}. Let us remark that in \cite{DMS09}, the existence and uniqueness result has been established but not It\^o's formula.
\begin{lemma}\label{Itospdeweaker}
Under the assumptions {\bf(H)} and {\bf(HI\#)}, the SPDE (\ref{SPDEO}) without obstacle admits a unique solution $u\in \cH_T$.  Moreover, it satisfies It\^o's formula i.e.
if $\varphi:\bbR\rightarrow\bbR$ is a function of class $C^2$ such that $\varphi''$ is bounded and $\varphi'(0)=0$, then the following relation holds almost surely, for all $t\in [0,T]$,
\begin{eqnarray}\label{It\^o'sspdeweaker}&&
\int_{{\cal O}}\varphi \left( u_t\left( x\right) \right)
dx+\int_0^t{\cal E} \left( \varphi ^{\prime }\left( u_s\right)
,u_s\right) ds=\int_{{\cal O} }\varphi \left( \xi \left( x\right)
\right) dx+\int_0^t\left( \varphi ^{\prime }\left( u_s\right)
,f_s \left( u_s ,\nabla u_s \right)\right) ds\nonumber
\\&&-\int_0^t\sum_{i=1}^d\left(
\partial _i\left( \varphi ^{\prime }\left( u_s\right) \right)
,g_{i,s}(u_s,\nabla u_s \right) ds+\frac 12\int_0^t\left(
\varphi ^{\prime \prime }\left( u_s\right) ,\left| h_{s} (u_s
,\nabla u_s )\right| ^2\right) ds
\nonumber\\&&+\sum_{j=1}^{\infty}\int_0^t\left( \varphi ^{\prime }\left(
u_s\right) ,h_{j,s} (u_s ,\nabla u_s )\right)
dB_s^j .
\end{eqnarray}
\end{lemma}
The following lemma will be helpful in showing that the solution to problem (\ref{SPDEO}) is quasi-continuous.
\begin{lemma}\label{wquasicontinue}The  following  PDE with random coefficient $f^0$ and zero Dirichlet boundary condition
\begin{equation} \label{eq:w}
\left\{ \begin{split}
         &dw_t+Aw_tdt=f_t^0dt\\
                  & w_0=0
                          \end{split} \right.
                          \end{equation}
 has a unique solution $w\in\cH_T$. Moreover, $w$ admits a quasi-continuous version.
\end{lemma}
\begin{proof}[Proof of Theorem \ref{2estimate}.]  We split the proof in 2 steps:

\vspace{0.2cm}

\textbf{Step 1.} We prove an existence and uniqueness result for the problem (\ref{SPDEO}) under the stronger conditions {\bf (H)}, {\bf(O)}, {\bf (HI2)} and
{\bf (HO\#)}. The idea of the proof is the same as the proof of Theorem 4 in  \cite{DMZ12}. \\
We begin with the linear case i.e. we assume that $f$, $g$ and $h$ do not depend on $(u,\nabla u)$, this implies that $f=f^0$, $g=g^0$ and $h=h^0$. We consider the
following penalized equation:
$$d(u_t^n-S'_t)=L(u_t^n-S'_t)dt+\bar{f}_tdt+\sum_{i=1}^d\partial_i\bar{g}^i_tdt+\sum_{j=1}^{+\infty}\bar{h}^j_tdB^j_t+n(u_t^n-S_t)^-dt$$
where $\bar{f}=f-f'$, $\bar{g}=g-g'$and $\bar{h}=h-h'$.
Applying It\^o's formula (\ref{It\^o'sspdeweaker}) to $(u^n-S')^2$, we
have almost surely for all $t\in[0,T]$:
\begin{eqnarray*}\parallel
u_t^n-S'_t\parallel^2+2\int_0^t\mathcal{E}(u_s^n-S'_s)ds&=&\parallel\xi-S'_0\parallel^2+2\int_0^t((u_s^n-S'_s),\bar{f}_s)ds\\&-&2\sum_{i=1}^d\int_0^t(\partial_i(u_s^n-S'_s),\bar{g}^i_s)ds+2\sum_{j=1}^{+\infty}\int_0^t((u_s^n-S'_s),\bar{h}^j_s)dB^j_s
\\&+&2\int_0^t\int_\mathcal{O}(u_s^n-S'_s)n(u_s^n-S_s)^-ds+\int_0^t\parallel|\bar{h}_s|\parallel^2ds
.\end{eqnarray*} We remark
first that:\begin{eqnarray*}&&\int_0^t\int_\mathcal{O}(u_s^n-S'_s)n(u_s^n-S_s)^-ds=\int_0^t\int_\mathcal{O}(u_s^n-S_s+S_s-S'_s)n(u_s^n-S_s)^-ds\\&&=
-\int_0^t\int_\mathcal{O}n((u_s^n-S_s)^-)^2ds+\int_0^t\int_\mathcal{O}(S_s-S'_s)n(u_s^n-S_s)^-dxds.\end{eqnarray*}
The last term in the right member is non-positive because  $S_t\leq
S'_t$, thus,\begin{eqnarray*}\parallel
u_t^n-S'_t\parallel^2&+&2\int_0^t\mathcal{E}(u_s^n-S'_s)ds+2\int_0^tn\parallel(u_s^n-S_s)^-\parallel^2ds\leq\parallel\xi-S'_0\parallel^2\\&+&
2\int_0^t(u_s^n-S'_s,\bar{f}_s)ds-2\sum_{i=1}^d\int_0^t(\partial_i(u_s^n-S'_s),\bar{g}^i_s)ds\\&+&
2\sum_{j=1}^{+\infty}\int_0^t(u_s^n-S'_s,\bar{h}^j_s)dB^j_s
+\int_0^t\parallel|\bar{h}_s|\parallel^2ds,\ \ \ a.s.\end{eqnarray*} Then,
 Hölder's duality inequality (\ref{dual2}) and the relation (\ref{sobolev})
lead to the following estimates, for all $t$ in $[0,T]$, for any $\delta$, $\epsilon>0$,
\begin{eqnarray*}2|\int_0^t(u_s^n-S'_s,\bar{f}_s)ds|&\leq&\delta\left\|u^n-S'\right\|^2_{\#;T}+C_\delta\left(\left\|\bar{f}\right\|_{\#;T}^*\right)^2\\&\leq&C\delta\left(\left\|u^n-S'\right\|^2_{2,\infty;T}+\left\|\nabla(u^n-S')\right\|^2_{2,2;T}\right)
+C_\delta\left(\left\|\bar{f}\right\|_{\#;T}^*\right)^2,\end{eqnarray*}
and\begin{eqnarray*}2|\sum_{i=1}^d\int_0^t(\partial_i(u_s^n-S'_s),\bar{g}^i_s)ds|\leq\epsilon\left\|\nabla(u^n-S')\right\|_{2,2;T}^2+C_{\epsilon}\left\||\bar{g}|\right\|_{2,2;T}^2.\end{eqnarray*}
Moreover, thanks to the Burkholder-Davies-Gundy inequality, we get
\begin{eqnarray*}E\sup_{t\in[0,T]}|\sum_{j=1}^{+\infty}\int_0^t(u_s^n-S'_s,\bar{h}^j_s)dB^j_s|&\leq&c_1E[\int_0^T\sum_{j=1}^{+\infty}(u_s^n-S'_s,\bar{h}^j_s)^2ds]^{1/2}\\&\leq&
c_1E[\int_0^T\sum_{j=1}^{+\infty}\sup_{s\in[0,T]}\parallel
u_s^n-S'_s\parallel^2\parallel\bar{h}^j_s\parallel^2ds]^{1/2}\\&\leq&c_1E[\sup_{s\in[0,T]}\parallel
u_s^n-S'_s\parallel(\int_0^T\parallel|\bar{h}_s|\parallel^2ds)^{1/2}]\\&\leq&\epsilon
E\sup_{s\in[0,T]}\parallel
u_s^n-S'_s\parallel^2+\frac{c_1}{4\epsilon}E\int_0^T\parallel|\bar{h}_s|\parallel^2ds.\end{eqnarray*}Then
using the strict ellipticity  assumption and the inequalities above,
we get
\begin{eqnarray*}&&(1-2\epsilon-C\delta)E\sup_{t\in[0,T]}\parallel
u_t^n-S'_t\parallel^2+(2\lambda-\epsilon-C\delta)E\int_0^T\parallel\nabla(u_s^n-S'_s)\parallel^2ds
\\&&\leq
C(E\parallel\xi-S'_0\parallel^2+E(\parallel
\bar{f}\parallel_{\#;T}^*)^2+E\parallel
|\bar{g}|\parallel_{2,2;T}^2+E\parallel
|\bar{h}|\parallel_{2,2;T}^2).\end{eqnarray*} We  take $\epsilon$
and $\delta$ small enough such that $(1-2\epsilon-C\delta)>0$ and
$(2\lambda-\epsilon-C\delta)>0$,
\begin{eqnarray*}E\sup_{t\in[0,T]}\parallel u_t^n-S'_t\parallel^2+E\int_0^T\mathcal{E}(u_t^n-S'_t)dt\leq C.\end{eqnarray*}
Then, to prove the existence and uniqueness in this case, we can
follow line by line the proof based on a weak convergence argument
given in \cite{DMZ12}, Theorem 4. The only difference is that now
the estimates depend on $\| \bar{f}^0\|_{\# ;t}$ instead of $\|
\bar{f}^0\|_{2,2 ;t}$.

\vspace{0.2cm}

\textbf{Step 2.} Now we turn to the general case, i.e. assume {\bf(H)}, {\bf(O)}, {\bf
(HI\#)} and {\bf (HO\#)}.

We consider the following SPDE:
\begin{equation}\label{equ:w}
dw_t+Aw_tdt=f_t^0dt
\end{equation}
Thus $u-w$ satisfies the following OSPDE:
\begin{eqnarray*}
d(u_t-w_t)+A(u_t-w_t)dt&=&F_t(u_t-w_t,\nabla(u_t-w_t))dt+div G_t(u_t-w_t,\nabla(u_t-w_t))dt\\&+&H_t(u_t-w_t,\nabla(u_t-w_t))dB_t+\nu(x,dt),
\end{eqnarray*}
where $$F_t(x,y,z)=f_t(x,y+w,z+\nabla w)-f_t^0(x)$$ $$G_t(x,y,z)=g_t(x,y+w,z+\nabla w)$$ $$H_t(x,y,z)=h_t(x,y+w,z+\nabla w).$$
We can easily check that $F$, $G$ and $H$ satisfy the same Lipschitz conditions as $f$, $g$ and $h$ and also $F^0\in L^2(\Omega\times[0,T]\times\cO;\R)$,  $G^0\in L^2(\Omega\times[0,T]\times\cO;\R^d)$ and $H^0\in L^2(\Omega\times[0,T]\times\cO;\R^{\bbN^*})$.
Moreover, $u-w\geq S-w$ and $S-w\leq S'-w$ where $S'-w$ satisfies the following SPDE:
\begin{equation*}
d(S'_t-w_t)+A(S'_t-w_t)dt=(f'_t-f_t^0)dt+div g'_tdt+h'_tdB_t.
\end{equation*}
It is easy to see that $f'-f$, $g'$ and $h'$ satisfy {\bf (HO\#)}.
Therefore, from Step 1, we know that $(u-w,\nu)$ uniquely exists.\\Combining with the existence and uniqueness of $w$, we deduce that the solution of the problem (\ref{SPDEO}) uniquely exists under the weaker assumptions  {\bf (HI\#)} and {\bf (HO\#)}.\\And the quasi-continuity of $u$ comes from the quasi-continuity of $w$ and $u-w$.
\end{proof}

\subsection{Estimates of the positive part of the solution} 
To get the estimate of the solution of \eqref{SPDEO}, firstly, 
we establish an It\^o formula for $(u,\nu)$.
\begin{theorem}\label{Itosweaker} Let $(u,\nu)$ be the solution of OSPDE \eqref{SPDEO} and $\varphi:\mathbb{R}\rightarrow\mathbb{R}$ be a function of class $C^2$ and assume that $\varphi''$ is bounded and $\varphi'(0)=0$.
Then the following relation holds a.s. for all $t\in [0,T]$:
\begin{eqnarray*}&&
\int_{{\cal O}}\varphi \left( u_t\left( x\right) \right)
dx+\int_0^t{\cal E} \left( \varphi ^{\prime }\left( u_s\right)
,u_s\right) ds=\int_{{\cal O} }\varphi \left( \xi \left( x\right)
\right) dx+\int_0^t\left( \varphi ^{\prime }\left( u_s\right) ,f_s
\left( u_s ,\nabla u_s \right)\right) ds
\\&&-\int_0^t\sum_{i=1}^d\left(
\partial _i\left( \varphi ^{\prime }\left( u_s\right) \right)
,g_{i,s}(u_s,\nabla u_s \right) ds+\frac 12\int_0^t\left( \varphi
^{\prime \prime }\left( u_s\right) ,\left| h_{s} (u_s ,\nabla u_s
)\right| ^2\right) ds
\\&&+\sum_{j=1}^{\infty}\int_0^t\left( \varphi ^{\prime }\left(
u_s\right) ,h_{j,s} (u_s ,\nabla u_s )\right)
dB_s^j+\int_0^t\int_\cO\varphi'(u_s)\nu(dxds).
\end{eqnarray*}
\end{theorem}
\begin{proof} The idea is that we begin with the stronger case, where {\bf(H)}, {\bf(O)},
$\textbf{(HI2)}$ and $\textbf{(HO\#)}$ hold. In this case we have the It\^o's formula, see step 1 of the proof of Theorem \ref{2estimate}.
Then using an approximation argument we can obtain the It\^o's formula in the general case. More precisely:
\\We take the function
$f_n(\omega,t,x):=f(\omega,t,x,u,\nabla u)-f^0+f_n^0$, where
$f_n^0,\ n\in\bbN^*$, is a sequence of bounded functions such that
$E\left(\left\|f^0-f_n^0\right\|^*_{\#;t}\right)^2\rightarrow0,\ as\
n\rightarrow+\infty.$ We consider the following equation
\begin{equation*}
du_t^n(x)+Au_t^n(x)dt=f_t^n(x)dt+div\breve{g}_t(x)dt+\breve{h}_t(x)dB_t+\nu^n(x,dt)
\end{equation*}
where $\breve{g}(\omega,t,x)=g(\omega,t,x,u,\nabla u)$ and $\breve{h}(\omega,t,x)=h(\omega,t,x,u,\nabla u)$.
This is a linear equation in $u^n$ so from Theorem \ref{2estimate}, we know that $(u^n,\nu^n)$ uniquely exists. \\
Applying It\^o's formula for the difference of two solutions to $(u^n-u^m)^2$ (see Theorem 6 in \cite{DMZ12}),
we have, almost surely, for all $t\in[0,T]$, 
\begin{eqnarray*}\left\|u^n_t-u^m_t\right\|^2+2\int_0^t\cE(u^n_s-u^m_s)ds&=&2\int_0^t(u^n_s-u^m_s, f^n_s-f^m_s)ds\\&+&2\int_0^t\int_\cO(u^n_s-u^m_s)(\nu^n-\nu^m)(dxds).
\end{eqnarray*}
Remarking that
$$\int\int(u_n-u_m)(\nu_n-\nu_m)(dxds)=\int\int(S-u_m)\nu_n(dxds)-\int\int(u_n-S)\nu_m(dxds)\leq 0$$
and for $\delta>0$, we have
\begin{eqnarray*}
2\left|\int_0^t(u^n_s-u^m_s,f^n_s-f^m_s)ds\right|&\leq&\delta\left\|u^n-u^m\right\|_{\#;t}^2+C_\delta\left(\left\|f^n-f^m\right\|_{\#;t}^*\right)^2.
\end{eqnarray*}
Since $\cE(u^n-u^m)\geq\lambda\left\|\nabla(u^n-u^m)\right\|^2_2$,
we deduce that, for all $t\in[0,T]$, almost surely,
\begin{eqnarray}\label{unminusum}
\left\|u_t^n-u_t^m\right\|^2+2\lambda\left\|\nabla(u^n-u^m)\right\|^2_{2,2;t}\leq\delta\left\|u^n-u^m\right\|_{\#;t}^2+C_\delta\left(\left\|f^n-f^m\right\|_{\#;t}^*\right)^2
\end{eqnarray}
Taking the supremum and the expectation, we get
\begin{equation*}
E\left(\left\|u^n-u^m\right\|^2_{2,\infty;t}+\left\| \nabla
(u^n-u^m)\right\| _{2,2;t}^2\right)\leq\delta E\left\|u^n-u^m\right\|^2_{\#;t}+C_\delta E\left(\left\|f^n-f^m\right\|_{\#;t}^*\right)^2 .
\end{equation*}
Dominating the term $E\left\|u^n-u^m\right\| _{\#;t}^2$ by
using the estimate (\ref{sobolev}) and taking $\delta$ small enough, we obtain the following estimate:
\begin{eqnarray*}E\left(\left\|u_n-u_m\right\|^2_{2,\infty;t}+\left\|\nabla
(u_n-u_m)\right\|^2_{2,2;t}\right) \leq 2C_\delta E\left(\left\|f^n-f^m\right\|_{\#;t}^*\right)^2\rightarrow0,\ when\ n,\,m\rightarrow\infty
\end{eqnarray*}
Therefore, $(u^n)$ has a limit $u$ in $\cH_T$.\\ Now we want to find the limit of $(\nu^n)$: we denote by $v^n$ the parabolic potential associated to $\nu^n$,  and $z^n=u^n-v^n$, so $z^n$ satisfies the following SPDE
\begin{equation*}
dz_t^n (x) +Az_t^n (x)dt =f^n_t (x)dt-\sum_{i=1}^d \partial_i \breve{g}^{i}_t(x)dt
+ \sum_{j=1}^{\infty} \breve{h}^{j}_t(x)\, dB_t^j .
\end{equation*}
Applying It\^o's formula to $(z^{n}-z^m)^2$, doing the same calculus as before, we obtain the following relation:
\begin{eqnarray*}E\left(\left\|z_n-z_m\right\|^2_{2,\infty;t}+\left\|\nabla
(z_n-z_m)\right\|^2_{2,2;t}\right) \leq 2C_\delta E\left(\left\|f^n-f^m\right\|^*_{\#;t}\right)^2 \longrightarrow 0,\quad \ as\ n,\, m\rightarrow \infty . \end{eqnarray*}
As a consequence: $$
E\left(\left\|v_n-v_m\right\|^2_{2,\infty;t}+\left\|\nabla
(v_n-v_m)\right\|^2_{2,2;t}\right) \longrightarrow 0,\quad \ as\ n,\, m\rightarrow \infty .
$$
Therefore, $(v^n)$ has a limit $v$ in $\cH_T$. So, by extracting a subsequence, we can assume that $(v^n)$ converges to $v$ in $\mathcal{K}$, $P-$almost-surely. Then, it's clear that $v\in\cP$, and we denote by $\nu$ the regular random measure associated to the potential $v$. Moreover, we have $P-a.s.$
$\forall\varphi\in\cW_t^+$,
\begin{eqnarray*}\int_0^t\int_\cO\varphi(x,s)\nu(dxds)&=&\lim_{n\rightarrow\infty}\int_0^t\int_\cO\varphi(x,s)\nu^n(dxds)\\&=&\lim_{n\rightarrow\infty}\int_0^t-(v^n_s,\frac{\partial\varphi_s}{\partial s})ds+\int_0^t\cE(v^n_s,\varphi_s)ds\\&=&\int_0^t-(v_s,\frac{\partial\varphi_s}{\partial s})ds+\int_0^t\cE(v_s,\varphi_s)ds.\end{eqnarray*}
Hence, $(u^n,\nu^n)$ converges to $(u,\nu)$. Moreover, by Theorem
\ref{2estimate}, we know that the solution of problem (\ref{SPDEO})
uniquely exists and we apply It\^o's formula for
$(u^n,\nu^n)\,:$ $\forall t\in[0,T]$, 
\begin{eqnarray}\label{It\^o'sapproxim}&&
\int_{{\cal O}}\varphi \left( u_t^n\left( x\right) \right)
dx+\int_0^t{\cal E} \left( \varphi ^{\prime }\left( u_s^n\right)
,u_s^n\right) ds=\int_{{\cal O} }\varphi \left( \xi \left( x\right)
\right) dx+\int_0^t\left( \varphi ^{\prime }\left( u_s^n\right)
,f_s^n\right) ds
\nonumber\\&&-\int_0^t\sum_{i=1}^d\left(
\partial _i\left(\varphi ^{\prime }\left( u_s^n\right)\right)
,\breve{g}^i_s\right) ds+\frac 12\int_0^t\left(
\varphi ^{\prime \prime }\left( u_s^n\right) ,\left| \breve{h}_{s}\right| ^2\right) ds
+\sum_{j=1}^{\infty}\int_0^t\left( \varphi ^{\prime }\left(
u_s^n\right) ,\breve{h}_s^j\right)
dB_s^j\nonumber\\&&+\int_0^t\int_\cO\varphi'(u_s^n)\nu^n(dxds)\quad\quad a.s.
\end{eqnarray}
Now, we pass to the limit as $n$ tends to $+\infty$. First, by using Lemma \ref{convergemeas} and Skohorod condition, we have 
\begin{eqnarray*}\int_0^t\int_\cO\varphi'(u^n_s)\nu^n(dxds)=\int_0^t\int_\cO\varphi'(S_s)\nu^n(dxds)\rightarrow\int_0^t\int_\cO\varphi'(S_s)\nu(dxds)=\int_0^t\int_\cO\varphi'(u_s)\nu(dxds).\end{eqnarray*}
Moreover,
\begin{eqnarray*}
&&\left|\int_0^t\left( \varphi ^{\prime }\left( u_s^n\right), f_s^n\right) ds-\int_0^t\left( \varphi ^{\prime }\left( u_s\right), f_s\right) ds\right|\\&\leq&
\left|\int_0^t\left( \varphi ^{\prime }\left( u_s^n\right)- \varphi ^{\prime }\left( u_s\right), f_s^n\right) ds\right|
+\left|\int_0^t\left(\varphi ^{\prime }\left( u_s\right), f_s^n-f_s\right) ds\right|\\&\leq&C\left\|u^n-u\right\|_{\#;t}\left\|f^n\right\|_{\#;t}^*+C\left\|u\right\|_{\#;t}\left\|f^n-f\right\|_{\#;t}^*.
\end{eqnarray*}
The relation (\ref{sobolev}) and the strong convergence of $(u^n)$ yield that $E\left\|u^n-u\right\|_{\#;t}\rightarrow0$, as $n\rightarrow\infty$. So, by extracting a subsequence, we can assume that the right member in the previous inequality tends to $0$ $P-$almost surely as $n$ tends to $+\infty$. 
So we have
$$\lim_{n\rightarrow +\infty}\int_0^t\left( \varphi ^{\prime }\left( u_s^n\right), f_s^n\right) ds=\int_0^t\left( \varphi ^{\prime }\left( u_s\right), f_s\right) ds.$$
The convergences of the other terms in (\ref{It\^o'sapproxim}) are easily deduced from the strong convergence of $(u^n)$ to $u$ in $\cH_T$ and then we deduce the desired formula.
\end{proof}
This yields  the  estimate of the $\cH_T$-norm of $u$ under {\bf (HI\#)}:
\begin{proposition}\label{estimation}
Under the same hypotheses and notations as in the previous theorem, we have:
\begin{eqnarray*}E\left(\left\|u\right\|^2_{2,\infty;t}+\left\|\nabla u\right\|^2_{2,2;t}\right)&\leq& k(t)E(\left\|\xi-S'_0\right\|^2_2+\left(\left\|\bar{f}^0\right\|^*_{\#;t}\right)^2+
\left\|\bar{g}^0\right\|^2_{2,2;t}+\left\|\bar{h}^0\right\|^2_{2,2;t}
\\&+&
\left\|S'_0\right\|^2_2+\left(\left\|f'\right\|^*_{\#;t}\right)^2+
\left\|g'\right\|^2_{2,2;t}+\left\|h'\right\|^2_{2,2;t})\end{eqnarray*}for
each $t\in[0,T]$, where $k(t)$ is a constant that only depends on
the structure constants and $t$.
\end{proposition}
\begin{proof}
Since $(u-S',\nu)=\cR(\xi-S'_0,\bar{f},\bar{g},\bar{h},S-S')$, applying the above It\^o formula to $(u-S')^2$,we have, almost surely, for all $t\in [0,T]$:
\begin{eqnarray}\label{7}&&\left\|u_t-S'_t\right\|^2+2\int_0^t\mathcal{E}(u_s-S'_s)ds=\left\|\xi-S'_0\right\|^2+2\int_0^t(u_s-S'_s,\bar{f}_s(u_s-S'_s,\nabla (u_s-S'_s)))ds\nonumber\\&&-2\int_0^t(\nabla(u_s-S'_s),\bar{g}_s(u_s-S'_s,\nabla(u_s-S'_s))ds+2\int_0^t(u_s-S'_s,\bar{h}_s(u_s-S'_s,\nabla(u_s-S'_s))dB_s\nonumber\\&&+\int_0^t\left\|\bar{h}_s(u_s-S'_s,\nabla(u_s-S'_s))\right\|^2ds+2\int_0^t\int_\mathcal{O}(u_s-S'_s)(x)\nu(dxds).\end{eqnarray}
Remarking the following relation$$\int_0^t\int_\mathcal{O}(u_s-S'_s)\nu(dxds)\leq\int_0^t\int_\mathcal{O}(u_s-S_s)\nu(dxds)=0.$$
The Lipschitz conditions in $\bar{g}$ and $\bar{h}$ and Cauchy-Schwarz's inequality lead the following relations: for $\delta$, $\epsilon>0$, we have
\begin{eqnarray*}
\int_0^t(\nabla
(u_s-S'_s),\bar{g}_s(u_s-S'_s,\nabla(u_s-S'_s))ds&\leq&(\alpha+\epsilon)\left\|\nabla
(u-S')\right\|^2_{2,2;t}\\&&+c_\epsilon\left\|u-S'\right\|^2_{2,2;t}+c_\epsilon\left\|\bar{g}^0\right\|
_{2,2;t}^2,\end{eqnarray*} and
\begin{eqnarray*}
\int_0^t\left\|\bar{h}_s(u_s-S'_s,\nabla(u_s-S'_s))\right\|^2ds\leq(\beta^2+\epsilon)\left\|\nabla(u-S')\right\|^2_{2,2;t}+c_\epsilon\left\|u-S'\right\|^2_{2,2;t}+c_\epsilon\left\|\bar{h}^0\right\|^2_{2,2;t} .
\end{eqnarray*}
Moreover, the Lipschitz condition in $\bar{f}$, the duality relation between elements in  $L_{\#;t}$ and $L_{\#;t}^*$ \eqref{dual2} 
and Young's inequality \eqref{young} yield the following relation: 
\begin{eqnarray*}\int_0^t\left(u_s-S'_s,\bar{f}_s\left( u_s-S'_s,\nabla (u_s-S'_s)\right) \right)
ds&\leq&\epsilon \left\| \nabla (u-S')\right\| _{2,2;t}^2+c_\epsilon
\left\|u-S'\right\| _{2,2;t}^2\\&&+\delta \left\|u-S'\right\|
_{\#;t}^2+c_\delta \left( \left\| \bar{f}^0\right\|
^*_{\#;t}\right) ^2 , \end{eqnarray*}
Since $\cE(u-S')\geq\lambda\left\|\nabla(u-S')\right\|^2_2$, we deduce from \eqref{7} that for all $t\in[0,T]$, almost surely,
\begin{eqnarray}\label{8}
&&\left\|u_t-S'_t\right\|_2^2+2\left( \lambda -\alpha -\frac{\beta
^2}2-\frac 52\epsilon \right) \left\| \nabla (u-S')\right\|
_{2,2;t}^2\le \left\|\xi-S'_0\right\|_2^2+\delta \left\|u-S'\right\| _{\#;t}^2\nonumber\\&&+2c_\delta \left( \left\|\bar{f}^0\right\| ^*_{\#;t}\right)
^2+2c_\epsilon \left\|\bar{g}^0\right\|
_{2,2;t}^2+c_\epsilon \left\|\bar{h}^0\right\|
_{2,2;t}^2+5c_\epsilon \left\|u-S'\right\| _{2,2;t}^2+2M_t,
\end{eqnarray}
where $M_t:=\sum_{j=1}^{\infty}\int_0^t\left(u_s-S'_s,\bar{h}_s^j\left(
u_s-S'_s,\nabla (u_s-S'_s)\right) \right) dB_s^j$ represents the martingale part.
Further, using a stopping procedure while taking the expectation, the martingale part vanishes, so that we get
\begin{eqnarray*}
&&E\left\|u_t-S'_t\right\|_2^2+2\left( \lambda -\alpha -\frac{\beta
^2}2-\frac 52\epsilon \right) E\left\| \nabla (u-S')\right\|
_{2,2;t}^2\le E\left\|\xi-S'_0\right\|_2^2+\delta E\left\|u-S'\right\| _{\#;t}^2\\&&+2c_\delta E\left( \left\|\bar{f}^0\right\| ^*_{\#;t}\right)
^2+2c_\epsilon E\left\|\bar{g}^0\right\|
_{2,2;t}^2+c_\epsilon E\left\|\bar{h}^0\right\|
_{2,2;t}^2+5c_\epsilon \int_0^tE\left\|u_s-S'_s\right\|^2_2ds .
\end{eqnarray*}
Then we choose $\epsilon=\frac{1}{5}\left(\lambda-\alpha-\frac{\beta^2}{2}\right)$, set $\gamma=\lambda-\alpha-\frac{\beta^2}{2}$ and apply Gronwall's lemma obtaining
\begin{equation}\label{star}
E\left\|u_t-S'_t\right\|^2_2+\gamma E\left\|\nabla(u-S')\right\|^2_{2,2;t}\leq\left(\delta E\left\|u-S'\right\|^2_{\#;t}+ E\left[F\left(\delta, \xi-S'_0,\bar{f}^0,\bar{g}^0,\bar{h}^0,t\right)\right]\right)e^{5c_\epsilon t}
\end{equation}
where $F\left(\delta, \xi-S'_0,\bar{f}^0,\bar{g}^0,\bar{h}^0,t\right)=\left(\left\|\xi-S'_0\right\|^2+2c_\delta\left( \left\|\bar{f}^0\right\| ^*_{\#;t}\right)
^2+2c_\epsilon\left\|\bar{g}^0\right\|_{2,2;t}^2+c_\epsilon\left\|\bar{h}^0\right\|_{2,2;t}^2\right)$.
As a consequence one gets
\begin{equation}\label{twostar}
E\left\|u-S'\right\|^2_{2,2;t}\leq\frac{1}{5c_\epsilon}\left(\delta E\left\|u-S'\right\|^2_{\#;t}+E\left[F\left(\delta, \xi-S'_0,\bar{f}^0,\bar{g}^0,\bar{h}^0,t\right)\right]\right)\left(e^{5c_\epsilon t}-1\right) .
\end{equation}
Now we return to the inequality (\ref{8}) and take the supremum, getting
\begin{equation}\label{supremum}
\left\|u-S'\right\|^2_{2,\infty;t}\leq\delta\left\|u-S'\right\|^2_{\#;t}+F\left(\delta, \xi-S'_0,\bar{f}^0,\bar{g}^0,\bar{h}^0,t\right)+5c_\epsilon\left\|u-S'\right\|^2_{2,2;t}+2\sup_{s\leq t}M_s
\end{equation}
We would like to take the expectation in this relation and for that reason we need to estimate the bracket of the martingale part,
\begin{eqnarray*} \left\langle M\right\rangle _t^{\frac 12}&\le &\left\| u-S'\right\|
_{2,\infty ;t}\left\|\bar{h}(u-S',\nabla(u-S'))\right\| _{2,2;t}\\&\le &
\eta \left\| u-S'\right\| _{2,\infty ;t}^2+c_\eta \left( \left\|
u-S'\right\| _{2,2;t}^2+\left\| \nabla (u-S')\right\| _{2,2;t}^2+\left\|\bar{h}^0\right\|^2_{2,2;t}\right) \end{eqnarray*} with $\eta $ another small
parameter to be properly chosen. Using this estimate and the
inequality of Burkholder-Davis-Gundy we deduce from the inequality (\ref{supremum}):
$$ \left( 1-2C_{BDG}\eta \right) E\left\|u-S'\right\| _{2,\infty
;t}^2\le \delta E\left\|u-S'\right\| _{\#;t}^2+E\left[F\left(\delta, \xi-S'_0,\bar{f}^0,\bar{g}^0,\bar{h}^0,t\right)\right]$$
$$ +\left( 5c_\varepsilon +2C_{BDG}c_\eta
\right) E\left\|u-S'\right\| _{2,2;t}^2+2C_{BDG}c_\eta E\left\| \nabla
(u-S')\right\| _{2,2;t}^2+2C_{BDG}c_\eta E\left\|\bar{h}^0\right\|^2_{2,2;t}$$ where $C_{BDG}$ is the constant corresponding to the Burkholder-Davis-Gundy inequality.
Further we choose the parameter $\eta =\frac 1{4C_{BDG}}$
and combine this estimate with \eqref{star} and \eqref{twostar} to deduce an estimate of the form:%
\begin{eqnarray*}
E\left( \left\| u-S'\right\| _{2,\infty ;t}^2+\left\| \nabla
(u-S')\right\| _{2,2;t}^2\right) &\le& \delta c_2\left( t\right) E\left\|
u-S'\right\| _{\#;t}^2\\&+&c_3(\delta,t)E\left[R\left(\delta, \xi-S'_0,\bar{f}^0,\bar{g}^0,\bar{h}^0,t\right)\right]
\end{eqnarray*}
where $R\left(\delta, \xi-S'_0,\bar{f}^0,\bar{g}^0,\bar{h}^0,t\right)=\left(\left\|\xi-S'_0\right\|^2+\left( \left\|\bar{f}^0\right\| ^*_{\#;t}\right)
^2+\left\|\bar{g}^0\right\|_{2,2;t}^2+\left\|\bar{h}^0\right\|_{2,2;t}^2\right)$ and $c_3(\delta,t)$ is a
constant that depends on $\delta $ and $t,$ while $c_2\left( t\right) $ is independent of $
\delta .$ Dominating the term $E\left\|u-S'\right\| _{\#;t}^2$ by
using the estimate (\ref{sobolev}) and then choosing $\delta =\frac
1{2c_1^2c_2\left( t\right) }$,
we get the following estimate:
\begin{eqnarray*}
E\left(\left\|u-S'\right\|^2_{2,\infty;t}+\left\|\nabla(u-S')\right\|^2_{2,2;t}\right)\leq k(t)E\left(\left\|\xi-S'_0\right\|^2_2+\left(\left\|\bar{f}^0\right\|^*_{\#;t}\right)^2+
\left\|\bar{g}^0\right\|^2_{2,2;t}+\left\|\bar{h}^0\right\|^2_{2,2;t}\right).
\end{eqnarray*}
Combining with the estimate for $S'$ (see Remark \ref{remark3}), we obtain the estimate asserted by our proposition.
\end{proof}
In the following Proposition, we establish a crucial relation for the positive part of $u$:
\begin{proposition}\label{Ito2positive}
Under the hypotheses of Theorem \ref{Itosweaker} with same
notations, the following relation holds a.s. for all $t\in [0,T]$:
\begin{eqnarray*}\int_\mathcal{O}(u^+_t(x))^2dx+2\int_0^t\mathcal{E}(u^+_s)ds&=&\int_\mathcal{O}(\xi^+(x))^2dx+2\int_0^t(u^+_s,f_s(u_s,\nabla u_s))ds
\\&-&2\int_0^t(\nabla
u^+_s,g_s(u_s,\nabla u_s))ds+2\int_0^t(u^+_s,h_s(u_s,\nabla
u_s))dB_s\\&+&\int_0^t\left\| \1_{\{u_s>0\}}h_s(u_s,\nabla
u_s)\right\|^2ds+2\int_0^t\int_\mathcal{O}u^+_s(x)\nu(dxds).\end{eqnarray*}
\end{proposition}
\begin{proof} We approximate $\psi(y)=(y^+)^2$ by a sequence of regular
functions: Let $\varphi $ be an increasing ${\cal C}^\infty $
function such that $\varphi \left( y\right) =0$ for any $y\in
]-\infty ,1]$ and $\varphi \left( y\right) =1$ for any $y\in
[2,\infty [.$ We set $\psi _n\left( y\right) =y^2\varphi \left(
ny\right) ,$ for each $y\in \mathbb{{R}}$ and all $n\in
\mathbb{{N}}^{*}.$ It is easy to verify that $\left( \psi _n\right)
_{n\in \mathbb{{ N}}^{*}}$ converges
uniformly to the function $\psi $ and that%
$$ \lim _{n\rightarrow \infty }\psi _n^{\prime }\left( y\right)
=2y^{+},\;\lim _{n\rightarrow \infty }\psi _n^{\prime \prime }\left(
y\right) =2\cdot \1_{\left\{ y>0\right\} }, $$
for any $y\in \mathbb{{R}}.$ Moreover we have the estimates%
\begin{equation}\label{varphiprime} 0\le \psi _n\left( y\right) \le \psi \left( y\right) ,\;0\le
\psi _n^{\prime }\left( y\right) \le Cy,\;\left| \psi _n^{\prime
\prime }\left( y\right) \right| \le C, \end{equation}
for any $y\ge 0$ and all
$n\in \mathbb{{N}}^{*},$ where $C$ is a constant. We have for all $n\in \mathbb{{ N}}^{*}$ and each $t\in[0,T],$ a.s.,%
\begin{equation}
\label{psi-n}
\begin{split}
& \int_{{\cal O}}\psi _n\left( u_t\left( x\right) \right) dx+\int_0^t{\cal E}%
\left( \psi _n{}^{\prime }\left( u_s\right) ,u_s\right) ds=\int_{{\cal O}%
}\psi _n\left( \xi \left( x\right) \right) dx+\int_0^t\left( \psi
_n{}^{\prime }\left( u_s\right) ,f_s\left( u_s,\nabla u_s\right)
\right) ds \\
&  -\int_0^t\sum_{i=1}^d\left( \psi _n{}^{\prime \prime }\left(
u_s\right)
\partial _iu_s,g_{i,s}\left( u_s,\nabla u_s\right) \right) ds+\frac
12\int_0^t\left( \psi _n{}^{\prime \prime }\left( u_s\right) ,\left|
h_s\left( u_s,\nabla u_s\right) \right| ^2\right) ds \\
& \quad \quad  +\sum_{j=1}^{\infty}\int_0^t\left( \psi _n{}^{\prime
}\left( u_s\right) ,h_{j,s}\left( u_s,\nabla u_s\right) \right)
dB_s^j+\int_0^t\int_\cO\psi _n{}^{\prime }\left(
u_s\right)\nu(dxds).
\end{split}
\end{equation}
Taking the limit, thanks to the dominated convergence theorem, we
know that all the terms except $\int_0^t\int_\cO\psi _n{}^{\prime
}\left( u_s\right)\nu(dxds)$ converge. From (\ref{varphiprime}) and (\ref{psi-n}), it is easy to verify $$\sup_n\int_0^t\int_\cO\psi'_n(u_s)\nu(dxds)\leq C.$$
Then, by Fatou's lemma, we have$$\int_0^t\int_\cO u^+_s(x)\nu(dxds)=\liminf_{n\rightarrow\infty}\int_0^t\int_\cO\psi'_n(u_s)\nu(dxds) <+\infty,\qquad a.s.$$
Hence, the convergence of the last term
comes from the dominated convergence theorem.
\end{proof}
Now we prove an estimate for the positive part $u^{+}$ of the
solution.
 For this we need the following notations:%
\begin{equation}
\label{f0+} \begin{split} &\bar{f}^{u-S',0}=\1_{\left\{
u>S'\right\} }\bar{f}^0,\;\bar{g}^{u-S',0}=\1_{\left\{
u>S'\right\} }\bar{g}^0,\;\bar{h}^{u-S',0}=\1_{\left\{
u>S'\right\} }\bar{h}^0,\\
&\bar{f}^{u-S'}=\bar{f}-\bar{f}^0+\bar{f}^{u-S',0},\;\bar{g}^{u-S'}=\bar{g}-\bar{g}^0+\bar{g}^{u-S',0},\;\bar{h}^{u-S'}=\bar{h}-\bar{h}^0+\bar{h}^{u-S',0}, \\
& \bar{f}^{u-S',0+}=\1_{\left\{ u>S'\right\} }\left(
\bar{f}^0\vee 0\right) ,\;(\xi-S'_0)
^{+}=(\xi-S'_0)\vee 0.\\
\end{split}
\end{equation}
\begin{proposition} Under the hypotheses of Proposition \ref{Ito2positive}, one has the following estimate:
\begin{eqnarray*} E\left(\left\|u^+\right\|^2_{2,\infty;t}\right)&\leq&2k(t)E\big(\left\|(\xi-S'_0)^+\right\|^2_2+\left(\left\|\bar{f}^{u-S',0+}\right\|^*_{\#;t}\right)^2+
\left\|\bar{g}^{u-S',0}\right\|^2_{2,2;t}+\left\|\bar{h}^{u-S',0}\right\|^2_{2,2;t}\\&+&\left\|
S^{'+}_0\right\|_2^2+\left( \left\|f^{',0+}\right\|_{\#;t}^*\right)
^2+\left\|g^{',0}\right\|_{2,2;t}^2+\left\|h^{',0}\right\|
_{2,2;t}^2\big)
\end{eqnarray*}
for each $t\in[0,T]$, where $k(t)$ is a constant that only depends
on the structure constants and $t$.
\end{proposition}
\begin{proof}
Since $(u-S',\nu)=\mathcal{R}(\xi-S'_0,\bar{f},\bar{g},\bar{h},S-S')$, by Proposition \ref{Ito2positive},
we have almost surely $\forall t\in[0,T]$:
\begin{eqnarray*}&&\ \ \int_\mathcal{O}((u_t-S'_t)^+(x))^2dx+2\int_0^t\mathcal{E}((u_s-S'_s)^+)ds\\&&=\int_\mathcal{O}((\xi-S'_0)^+(x))^2dx+2\int_0^t((u_s-S'_s)^+,\bar{f}_s(u_s-S'_s,\nabla (u_s-S'_s)))ds
\\&&-2\int_0^t(\nabla
(u_s-S'_s)^+,\bar{g}_s(u_s-S'_s,\nabla(
u_s-S'_s))ds+2\int_0^t((u_s-S'_s)^+,\bar{h}_s(u_s-S'_s,\nabla
(u_s-S'_s))dB_s\\&&+\int_0^t\left\|
\1_{\{u_s-S'_s>0\}}(\bar{h}_s(u_s-S'_s,\nabla(
u_s-S'_s)))\right\|^2ds+2\int_0^t\int_\mathcal{O}(u_s-S'_s)^+(x)\nu(dxds).\end{eqnarray*}
As the support of $\nu$ is $\{u=S\}$, we have the following relation
$$\int_0^t\int_\mathcal{O}(u_s-S'_s)^+\nu(dxds)=\int_0^t\int_\mathcal{O}(S_s-S'_s)^+\nu(dxds)=0.$$
Then we repeat word by word the proof of Proposition \ref{estimation}, replacing $u-S'$, $\bar{f}$, $\bar{g}$, $\bar{h}$ and $\xi-S'_0$ by $(u-S')^+$, $\bar{f}^{u-S',0+}$, $\bar{g}^{u-S',0}$, $\bar{h}^{u-S',0}$ and $(\xi-S'_0)^+$ respectively.
Hence, we get the following estimate:
\begin{eqnarray*}E(\left\|(u-S')^+\right\|^2_{2,\infty;t}&+&\left\|\nabla(u-S')^+\right\|^2_{2,2;t})\leq k(t)E\big(\left\|(\xi-S'_0)^+\right\|^2_2\\&&\qquad\qquad\qquad+\left(\left\|\bar{f}^{u-S',0+}\right\|^*_{\#;t}\right)^2+
\left\|\bar{g}^{u-S',0}\right\|^2_{2,2;t}+\left\|\bar{h}^{u-S',0}\right\|^2_{2,2;t}\big).
\end{eqnarray*}
Moreover, from Theorem 4 in \cite{DMS09}, we know that
\begin{eqnarray*}E\left(\left\| S^{'+}\right\| _{2,\infty ;t}^2+\left\|
\nabla S^{'+}\right\| _{2,2;t}^2\right)\leq k\left( t\right)E\left(
\left\| S^{'+}_0\right\|_2^2+\left( \left\|f^{',0+}\right\|
_{\#;t}^*\right) ^2+\left\|g^{',0}\right\|
_{2,2;t}^2+\left\|h^{',0}\right\| _{2,2;t}^2\right).
\end{eqnarray*}
where $S^{'+}_0=S'_0\vee0$, $f^{',0+}=\1_{\{S'>0\}}(f'\vee0)$,
$g^{',0}=\1_{\{S'>0\}}g'$ and $h^{',0}=\1_{\{S'>0\}}h'$. Then with the
relation:
$$E\left\|(u)^+\right\|^2_{2,\infty;t}\leq 2E\left(\left\|(u-S')^+\right\|^2_{2,\infty;t}+\left\|(S')^+\right\|^2_{2,\infty;t}\right),$$
we get the desired estimate.
\end{proof}

\section{Comparison theorem}
In this section we will establish a comparison theorem for the solutions of two OSPDE, 
$(u^1,\nu^1)=\cR^{\#}(\xi^1,f^1,g,h,S^1)$ and $(u^2,\nu^2)=\cR^{\#}(\xi^2,f^2,g,h,S^2)$, 
where $(\xi^i,f^i,g,h,S^i)$ satisfy assumptions {\bf(H)}, {\bf(O)}, {\bf
(HI\#)} and {\bf (HO\#)}. It is obvious that $(u^1-S',\nu^1)=\cR^{\#}(\xi^1-S'_0,\bar{f}^1,\bar{g},\bar{h},S^1-S')$ 
and $(u^2-S',\nu^2)=\cR^{\#}(\xi^2-S'_0,\bar{f}^2,\bar{g},\bar{h},S^2-S')$, where $\bar{f}^i$, $\bar{g}$, $\bar{h}$
are defined as in \eqref{BAR} and $(\xi^i-S'_0,\bar{f}^i,\bar{g},\bar{h}, S^i-S')$ satisfy assumptions {\bf(H)}, {\bf(O)}, {\bf(HI\#)} and {\bf (HO2)}. 

\begin{theorem} \label{comparison}Assume that the following conditions
hold\begin{enumerate}
      \item $\xi^1\leq\xi^2,\ dx\otimes dP-a.e.$
      \item $f^1(u^1,\nabla u^1)\leq f^2(u^1,\nabla u^1),\ dt\otimes dx\otimes dP-a.e.$
      \item $S^1\leq S^2,\ dt\otimes dx\otimes dP-a.e.$
    \end{enumerate}
Then for almost all $\omega\in\Omega$, $u^1(t,x)\leq u^2(t,x),\
q.e.$\end{theorem}
\begin{proof} 
We take the functions: $\bar{f}^{1,n}(\omega,t,x):=\bar{f}^1(u^1-S',\nabla(u^1-S'))-\bar{f}^{1,0}+\bar{f}^{1,0}_n$ 
and $\bar{f}^{2,n}(\omega,t,x):=\bar{f}^2(u^2-S',\nabla(u^2-S'))-\bar{f}^{2,0}+\bar{f}^{2,0}_n$, 
where $\bar{f}^{1,0}_n$ and $\bar{f}^{2,0}_n,\ n\in\bbN^*$
are two sequences of bounded functions such that $E\big(\big\|\bar{f}^{1,0}-\bar{f}^{1,0}_n\big\|_{\#;t}^*\big)^2$ and 
$E\big(\big\|\bar{f}^{2,0}-\bar{f}^{2,0}_n\big\|_{\#;t}^*\big)^2$ tend to $0$ as $n\rightarrow\infty$. 
We consider the following two equations 
\begin{equation*}
d\bar{u}_t^{1,n}(x)+A\bar{u}_t^{1,n}(x)dt=\bar{f}^{1,n}_t(x)dt+div\breve{g}_t^1(x)dt+\breve{h}^1_t(x)dB_t+\nu^{1,n}(x,dt)\,,
\end{equation*}and
\begin{equation*}
d\bar{u}_t^{2,n}(x)+A\bar{u}_t^{2,n}(x)dt=\bar{f}^{2,n}_t(x)dt+div\breve{g}_t^2(x)dt+\breve{h}^2_t(x)dB_t+\nu^{2,n}(x,dt)\,,
\end{equation*}
where $\breve{g}^i(\omega,t,x)=\bar{g}(\omega,t,x,u^i-S',\nabla(u^i-S'))$ and $\breve{h}^i(\omega,t,x)=\bar{h}(\omega,t,x,u^i-S',\nabla(u^i-S'))$ for $i=1,2$.
From the proof of Theorem \ref{Itosweaker}, we know that $(\bar{u}^{i,n},\nu^{i,n})$ converge to $(\bar{u}^i,\nu^i)$. 
\\As $(\bar{u}^{i,n},\nu^{i,n})=\cR(\xi^i-S'_0,\bar{f}^{i,n},\breve{g}^i, \breve{h}^i,S^i-S')$ with $(\xi^i-S'_0,\bar{f}^{i,n},\breve{g}^i, \breve{h}^i,S^i-S')$
satisfy assumptions {\bf(H)}, {\bf(O)}, {\bf(HI2)} and {\bf (HO2)}, we have the following It\^o's formula for the difference $\hat{\bar{u}}^n:=\bar{u}^{1,n}-\bar{u}^{2,n}$ of the solutions of 
two OSPDE (see Theorem 6 in \cite{DMZ12}), for all $t\in[0,T]$ and any function $\varphi: \R\longrightarrow \R$ of class $C^2$ with bounded second order derivative such that $\varphi' (0)=0$ : 
\begin{eqnarray}\label{itoito}&&\int_\mathcal{O}\varphi(\hat{\bar{u}}^n_t(x))dx+\int_0^t\mathcal{E}(\varphi'(\hat{\bar{u}}^n_s),\hat{\bar{u}}^n_s)ds
=\int_\cO\varphi(\hat{\xi}(x))dx+\int_0^t(\varphi'(\hat{\bar{u}}^n_s),\hat{\bar{f}}^n_s)ds
\\&&-\sum_{k=1}^d\int_0^t\int_\mathcal{O}\varphi''(\hat{\bar{u}}^n_s(x))
\partial_k(\hat{\bar{u}}^n_s(x))\hat{\bar{g}}^k_s(x)dxds+\sum_{j=1}^{+\infty}\int_0^t\int_\cO\varphi'(\hat{\bar{u}}^n_s(x))\hat{\bar{h}}^j_s(x)dxdB^j_s
\nonumber\\&&+\frac{1}{2}\sum_{j=1}^{+\infty}\int_0^t\int_\mathcal{O}\varphi''(\hat{\bar{u}}^n_s(x))(\hat{\bar{h}}^j_s(x))^2dxds
+\int_0^t\int_\mathcal{O}\varphi'(\hat{\bar{u}}^n_s(x))(\nu^{1,n}-\nu^{2,n})(dx, ds)\quad a.s.\nonumber
\end{eqnarray}
Then we approximate the function $\psi:\ y\in\mathbb{R}\rightarrow\varphi(y^+)$ by a sequence $(\psi_m)$ of regular functions. Let $\zeta$ be a $\mathcal{C}^\infty$ increasing function such that
$$\forall y\in]-\infty,1],\ \zeta(y)=0\ and\ \forall y\in[2,+\infty[,\ \zeta(y)=1.$$
We set for all $n$:$$\forall y\in\mathbb{R},\ \
\psi_m(y)=\varphi(y)\zeta(ny).$$ It is easy to verify that
$(\psi_m)$ converges uniformly to the function $\psi$, $(\psi'_m)$
converges everywhere to the function $(y\rightarrow\varphi'(y^+))$
and $(\psi''_m)$ converges everywhere to the function $(y\rightarrow
\1_{\{y>0\}}\varphi''(y^+))$. Moreover we have the estimates:
\begin{equation}\label{controlofpsi}
\forall y\in\mathbb{R}^+,\ m\in\mathbb{N}^*,\ \
0\leq\psi_m(y)\leq\psi(y),\ \ 0\leq\psi'_m(y)\leq Cy,\ \
\left|\psi_m''(y)\right|\leq C,
\end{equation}
where $C$ is a constant. Thanks to \eqref{itoito} we have
almost surely, for $t\in[0,T]$,
\begin{eqnarray*}&&\int_\mathcal{O}\psi_m(\hat{\bar{u}}^n_t(x))dx+\int_0^t\mathcal{E}(\psi_m'(\hat{\bar{u}}^n_s),\hat{\bar{u}}^n_s)ds
=\int_\cO\psi_m(\hat{\xi}(x))dx+\int_0^t(\psi_m'(\hat{\bar{u}}^n_s),\hat{\bar{f}}^n_s)ds
\\&&-\sum_{k=1}^d\int_0^t\int_\mathcal{O}\psi_m''(\hat{\bar{u}}^n_s(x))
\partial_k(\hat{\bar{u}}^n_s(x))\hat{\bar{g}}^k_s(x)dxds+\sum_{j=1}^{+\infty}\int_0^t\int_\cO\psi_m'(\hat{\bar{u}}^n_s(x))\hat{\bar{h}}^j_s(x)dxdB^j_s
\nonumber\\&&+\frac{1}{2}\sum_{j=1}^{+\infty}\int_0^t\int_\mathcal{O}\psi_m''(\hat{\bar{u}}^n_s(x))(\hat{\bar{h}}^j_s(x))^2dxds
+\int_0^t\int_\mathcal{O}\psi_m'(\hat{\bar{u}}^n_s(x))(\nu^{1,n}-\nu^{2,n})(dx, ds)\,.
\end{eqnarray*}
Making $m$ tend to $+\infty$ and using the fact that $\1_{\{\hat{\bar{u}}>0\}}\partial_k\hat{\bar{u}}_s=\partial_k\hat{\bar{u}}_s^+$, we get by the
dominated convergence theorem the following formula:
\begin{eqnarray*}&&\int_\mathcal{O}\varphi(\hat{\bar{u}}^+_{n,t}(x))dx+\int_0^t\mathcal{E}(\varphi'(\hat{\bar{u}}^+_{n,s}),\hat{\bar{u}}^+_{n,s})ds=\int_\cO\varphi(\hat{\xi}^+(x))dx+\int_0^t(\varphi'(\hat{\bar{u}}^+_{n,s}),\hat{\bar{f}}^n_s)ds
\\&&-\sum_{k=1}^d\int_0^t\int_\mathcal{O}\varphi''(\hat{\bar{u}}^+_{n,s}(x))
\partial_k(\hat{\bar{u}}^+_{n,s}(x))\hat{\bar{g}}^k_s(x)dxds+\sum_{j=1}^{+\infty}\int_0^t\int_\cO\varphi'(\hat{\bar{u}}^+_{n,s}(x))\hat{\bar{h}}^j_s(x)dxdB^j_s
\nonumber\\&&+\frac{1}{2}\sum_{j=1}^{+\infty}\int_0^t\int_\mathcal{O}\varphi''(\hat{\bar{u}}^+_{n,s}(x))\1_{\{\hat{\bar{u}}_{n,s}>0\}}|\hat{\bar{h}}^j_s(x)|^2dxds+\int_0^t\int_\mathcal{O}\varphi'(\hat{\bar{u}}^+_{n,s}(x))\hat{\nu}^n(dx, ds)\,,\ a.s.\nonumber
\end{eqnarray*}
Thanks to the strong convergence of $(\hat{\bar{u}}^n)_n$, by taking $n$ tend to $+\infty$, we obtain:
\begin{eqnarray*}&&\int_\mathcal{O}\varphi(\hat{\bar{u}}^+_{t}(x))dx+\int_0^t\mathcal{E}(\varphi'(\hat{\bar{u}}^+_{s}),\hat{\bar{u}}^+_{s})ds=\int_\cO\varphi(\hat{\xi}^+(x))dx+\int_0^t(\varphi'(\hat{\bar{u}}^+_{s}),\hat{\bar{f}}_s)ds
\\&&-\sum_{k=1}^d\int_0^t\int_\mathcal{O}\varphi''(\hat{\bar{u}}^+_{s}(x))
\partial_k(\hat{\bar{u}}^+_{s}(x))\hat{\bar{g}}^k_s(x)dxds+\sum_{j=1}^{+\infty}\int_0^t\int_\cO\varphi'(\hat{\bar{u}}^+_{s}(x))\hat{\bar{h}}^j_s(x)dxdB^j_s
\nonumber\\&&+\frac{1}{2}\sum_{j=1}^{+\infty}\int_0^t\int_\mathcal{O}\varphi''(\hat{\bar{u}}^+_{s}(x))\1_{\{\hat{\bar{u}}_{s}>0\}}|\hat{\bar{h}}^j_s(x)|^2dxds+\int_0^t\int_\mathcal{O}\varphi'(\hat{\bar{u}}^+_{s}(x))\hat{\nu}(dx, ds)\,,\ a.s.\nonumber
\end{eqnarray*}
Taking $\varphi(x)=x^2$, we have almost surely for all $t\in[0,T]$,
\begin{eqnarray*}&&\int_\mathcal{O}(\hat{\bar{u}}^+_{t}(x))^2dx+2\int_0^t\mathcal{E}(\hat{\bar{u}}^+_{s})ds=\int_\cO(\hat{\xi}^+(x))^2dx+2\int_0^t(\hat{\bar{u}}^+_{s},\hat{\bar{f}}_s)ds\\&&-2\sum_{k=1}^d\int_0^t\int_\mathcal{O}
\partial_k\hat{\bar{u}}^+_{s}(x)\hat{\bar{g}}^k_s(x)dxds+2\sum_{j=1}^{+\infty}\int_0^t\int_\cO\hat{\bar{u}}^+_{s}(x)\hat{\bar{h}}^j_s(x)dxdB^j_s
\nonumber\\&&+\sum_{j=1}^{+\infty}\int_0^t\int_\mathcal{O}\1_{\{\hat{\bar{u}}_{s}>0\}}|\hat{\bar{h}}^j_s(x)|^2dxds+2\int_0^t\int_\mathcal{O}\hat{\bar{u}}^+_{s}(x)\hat{\nu}(dx, ds)\,.\nonumber
\end{eqnarray*}
Remarking the following relation
\begin{eqnarray*}&&\int_0^t\int_\mathcal{O}\hat{\bar{u}}_s^+(x)\hat{\nu}(dxds)=\int_0^t\int_\cO(u_s^1-S'_s-(u^2_s-S'_s))^+\hat{\nu}(dxds)
=\int_0^t\int_\cO(u^1-u^2)^+\hat{\nu}(dxds)\\&&=\int_0^t\int_\mathcal{O}(S^1-u^2)^+\nu^1(dxds)
-\int_0^t\int_\mathcal{O}(u^1-S^2)^+\nu^2(dxds)\leq0 .
\end{eqnarray*}
Then a similar argument as in the proof of Proposition \ref{estimation} yields the following estimate: 
\begin{eqnarray*}
E\left(\left\|\hat{\bar{u}}^+\right\|^2_{2,\infty;t}+\left\|\nabla\hat{\bar{u}}^+\right\|^2_{2,2;t}\right)\leq
k(t)E\left(\left\|\hat{\xi}^+\right\|^2_2+\left(\left\|\hat{\bar{f}}^{\hat{\bar{u}},0+}\right\|^*_{\theta;t}\right)^2+\left\|\hat{\bar{g}}^{\hat{\bar{u}},0}\right\|^2_{2,2;t}+\left\|\hat{\bar{h}}^{\hat{\bar{u}},0}\right\|^2_{2,2;t}\right).
\end{eqnarray*}
This implies $\bar{u}^1\leq\bar{u}^2$ since $\hat{\xi}\leq0$, $\hat{\bar{f}}^0\leq0$
and $\hat{\bar{g}}^0=\hat{\bar{h}}^0=0$. Thanks to the uniqueness of the solution of OSPDE (see Theorem \ref{2estimate}), 
we know that $\bar{u}^1=u^1-S'$ and $\bar{u}^2=u^2-S'$ which yields the desired result. 
\end{proof}

\section{Appendix}
\subsection{Proof of Lemma \ref{Itospdeweaker}}
\begin{proof}
We take the function $f_n(\omega,t,x):=f(\omega,t,x,u,\nabla
u)-f^0+f_n^0$, where $f_n^0,\ n\in\bbN^*$, is a sequence of bounded
functions such that
$E\left(\left\|f^0-f_n^0\right\|^*_{\#;t}\right)^2\rightarrow0,\ as\
n\rightarrow+\infty.$ We consider the following equation
\begin{equation*}
du_t^n(x)+Au_t^n(x)dt=f_t^n(x)dt+div\breve{g}_t(x)dt+\breve{h}_t(x)dB_t
\end{equation*}
where $\breve{g}(\omega,t,x)=g(\omega,t,x,u,\nabla u)$ and $\breve{h}(\omega,t,x)=h(\omega,t,x,u,\nabla u)$.
This is a linear equation in $u^n$, from \cite{DenisStoica}, we know that $u^n$ uniquely exists. \\
Applying It\^o's formula to $(u^n-u^m)^2$, almost surely, for all $t\in[0,T]$, 
\begin{eqnarray*}
\left\|u^n_t-u^m_t\right\|^2+2\int_0^t\cE(u^n_s-u^m_s)ds&=&2\int_0^t(u^n_s-u^m_s, f^n_s-f^m_s)ds.
\end{eqnarray*}
From (\ref{dual2}), we have, for $\delta>0$,
\begin{eqnarray*}
2\left|\int_0^t(u^n_s-u^m_s,f^n_s-f^m_s)ds\right|&\leq&\delta\left\|u^n-u^m\right\|_{\#;t}^2+C_\delta\left(\left\|f^n-f^m\right\|_{\#;t}^*\right)^2.
\end{eqnarray*}
Since $\cE(u^n-u^m)\geq\lambda\left\|\nabla(u^n-u^m)\right\|^2_2$, we deduce that, for all $t\in[0,T]$, almost surely,
\begin{eqnarray}\label{unminusum}
\left\|u_t^n-u_t^m\right\|^2+2\lambda\left\|\nabla(u^n-u^m)\right\|^2_{2,2;t}\leq\delta\left\|u^n-u^m\right\|_{\#;t}^2+C_\delta\left(\left\|f^n-f^m\right\|_{\#;t}^*\right)^2.
\end{eqnarray}
Taking the supremum and the expectation, we get
\begin{equation*}
E\left(\left\|u^n-u^m\right\|^2_{2,\infty;t}+\left\| \nabla
(u^n-u^m)\right\| _{2,2;t}^2\right)\leq\delta E\left\|u^n-u^m\right\|^2_{\#;t}+C_\delta E\left(\left\|f^n-f^m\right\|_{\#;t}^*\right)^2.
\end{equation*}
Dominating the term $E\left\|u^n-u^m\right\| _{\#;t}^2$ by
using the estimate (\ref{sobolev}) and taking $\delta$ small enough, we obtain the following estimate:
\begin{eqnarray*}E\left(\left\|u_n-u_m\right\|^2_{2,\infty;t}+\left\|\nabla
(u_n-u_m)\right\|^2_{2,2;t}\right) \leq C_\delta E\left(\left\|f^n-f^m\right\|_{\#;t}^*\right)^2\rightarrow0,\ when\ n,\,m\rightarrow\infty .
\end{eqnarray*}
Therefore $(u^n)$ has a limit $u$ in $\cH$.
\\ See for example \cite{DMS05}, we know that for $u^n$ we have the following It\^o's formula, for all $t\in[0,T]$, $P-a.s.$: 
\begin{eqnarray*}&&
\int_{{\cal O}}\varphi \left( u_t^n\left( x\right) \right)
dx+\int_0^t{\cal E} \left( \varphi ^{\prime }\left( u_s^n\right)
,u_s^n\right) ds=\int_{{\cal O} }\varphi \left( \xi \left( x\right)
\right) dx+\int_0^t\left( \varphi ^{\prime }\left( u_s^n\right)
,f_s^n\right) ds
\\&&-\int_0^t\sum_{i=1}^d\left(
\partial _i\left(\varphi ^{\prime }\left( u_s^n\right)\right)
,\breve{g}^i_s\right) ds+\frac 12\int_0^t\left(
\varphi ^{\prime \prime }\left( u_s^n\right) ,\left| \breve{h}_{s}\right| ^2\right) ds
+\sum_{j=1}^{\infty}\int_0^t\left( \varphi ^{\prime }\left(
u_s^n\right) ,\breve{h}_s^j\right)
dB_s^j.
\end{eqnarray*}
Now, we pass to the limit as $n$ tend to $+\infty$:
\begin{eqnarray*}
&&\left|\int_0^t\left( \varphi ^{\prime }\left( u_s^n\right), f_s^n\right) ds-\int_0^t\left( \varphi ^{\prime }\left( u_s\right), f_s\right) ds\right|\\&\leq&
\left|\int_0^t\left( \varphi ^{\prime }\left( u_s^n\right)- \varphi ^{\prime }\left( u_s\right), f_s^n\right) ds\right|
+\left|\int_0^t\left(\varphi ^{\prime }\left( u_s\right), f_s^n-f_s\right) ds\right|\\&\leq&C\left\|u^n-u\right\|_{\#;t}\left\|f^n\right\|_{\#;t}^*+C\left\|u\right\|_{\#;t}\left\|f^n-f\right\|_{\#;t}^*.
\end{eqnarray*}
The relation (\ref{sobolev}) and the strong convergence of $(u^n)$ yield that $E\left\|u^n-u\right\|_{\#;t}\rightarrow0$, as $n\rightarrow\infty$. So, by extracting a subsequence, we can assume that the right member in the previous inequality tends to $0$ $P-$almost surely as $n$ tends to $+\infty$.
 So we have
$$\lim_{n\rightarrow +\infty}\int_0^t\left( \varphi ^{\prime }\left( u_s^n\right), f_s^n\right) ds=\int_0^t\left( \varphi ^{\prime }\left( u_s\right), f_s\right) ds.$$
The convergence of the other terms are easily deduced from the strong convergence of $(u^n)$ to $u$ in $\cH_T$ and yield the desired formula.
\end{proof}
\subsection{Proof of Lemma \ref{wquasicontinue}}
\begin{proof}First of all, this equation is a special case of Theorem 3 in \cite{DMS09} hence, we know that $w$ exists is unique and belongs to $\cH_T$.\\
Following M.Pierre \cite{Pierre, PIERRE} and F.Mignot and J.Puel \cite{MignotPuel}, we define
$$\kappa(w,0):=ess\inf\{u\in\cP;\ u\geq w\ a.e.\, ,\ u(0)\geq0\}.$$

We consider the following equation:
\begin{equation} \label{eq:1}
\left\{ \begin{split}
         &\frac{\partial v_t^n}{\partial t}=Lv_t^n+n(v_t^n-w_t)^-\\
                  & v^n_0=0
                          \end{split} \right.
                          \end{equation}
From \cite{MignotPuel}, for almost all $\omega\in\Omega$, we know that $v^n(\omega)$ converges weakly to $v(\omega):=\kappa^\omega(w,0)$ in $L^2(0,T;H_0^1(\cO))$ and that $v(\omega)\geq w(\omega)$.
\\(\ref{eq:w})-(\ref{eq:1}) yields
\begin{equation*}
d(v_t^n-w_t)+A(v_t^n-w_t)dt=(n(v_t^n-w_t)^--f_t^0)dt
\end{equation*}
so, we have the following relation almost surely, $\forall t\geq0$,
\begin{eqnarray*}
\left\|v_t^n-w_t\right\|^2+2\int_0^t\cE(v_s^n-w_s)ds=2\int_0^t\int_\cO(v_s^n-w_s)n(v_s^n-w_s)^-dxds-2\int_0^t(v_s^n-w_s,f_s^0)ds.
\end{eqnarray*}
The first term is negative and
\begin{eqnarray*}
\left|\int_0^t(v_s^n-w_s,f_s^0)ds\right|\leq\delta\left\|v^n-w\right\|^2_{\#;t}+C_\delta\left(\left\|f^0\right\|^*_{\#;t}\right)^2.
\end{eqnarray*}
Therefore
\begin{eqnarray*}
\left\|v_t^n-w_t\right\|^2_2+2\lambda\left\|\nabla(v^n-w)\right\|^2_{2,2;t}\leq 2\delta\left\|v^n-w\right\|^2_{\#;t}+2C_\delta\left(\left\|f^0\right\|^*_{\#;t}\right)^2.
\end{eqnarray*}
Taking the supremum and the expectation,  we get
\begin{eqnarray*}
E\left\|v^n-w\right\|^2_{2,\infty;t}\leq 2\delta E\left\|v^n-w\right\|^2_{\#;t}+2C_\delta E\left(\left\|f^0\right\|^*_{\#;t}\right)^2.
\end{eqnarray*}
Dominating the term $E\left\|v^n-w\right\| _{\#;t}^2$ by
using the estimate (\ref{sobolev}) and taking $\delta$ small enough, we obtain
\begin{equation*}
E\left\|v^n-w\right\|^2_{2,\infty;t}+E\left\|\nabla(v^n-w)\right\|^2_{2,2;t}\leq CE\left(\left\|f^0\right\|^*_{\#;t}\right)^2.
\end{equation*}
By Fatou's lemma, we have
\begin{equation}\label{estimw}
E\sup_{t\in[0,T]}\left\|\kappa_t-w_t\right\|^2+E\int_0^T\cE(\kappa_t-w_t)dt\leq CE\int_0^T\left(\left\|f^0_t\right\|^*_{\#}\right)^2dt.
\end{equation}
We now consider  a sequence of  random functions   $(f^{0,n})_{n \in \N^*} $ which belongs in $ L^2(\Omega)\otimes C_c^\infty(\bbR^+)\otimes C_c^\infty(\cO)$  and such that $E\left\|f^{0,n}-f^0\right\|_{\#;t}^*\rightarrow0$.
Let $w^n$  be the solution of
\begin{equation*}
\left\{ \begin{split}
         &dw_t^n+Aw_t^ndt=f_t^{0,n}dt\\
                  & w_0^n=0.
                          \end{split} \right.
\end{equation*}
Then it's well known that  $w^n$ is $P-$almost surely continuous in $(t,x)$ (see for example \cite{Aronson3}).
\\ Then, we consider a sequence of random open sets
$$\vartheta_n=\{|w^{n+1}-w^n|>\epsilon_n\},\quad\Theta_p=\bigcup_{n=p}^{+\infty}\vartheta_n,$$
and
$\kappa_n=\kappa(\frac{1}{\epsilon_n}(w^{n+1}-w^n),0)+\kappa(-\frac{1}{\epsilon_n}(w^{n+1}-w^n),0)$.
From the definition of $\kappa$ and the relation (see \cite{PIERRE}), we get
$$\kappa(|v|)\leq\kappa(v,v^+(0))+\kappa(-v,v^-(0).)$$
 Moreover,  $\kappa_n$ satisfy the conditions of Lemma 3.3 in \cite{PIERRE}, i.e. $\kappa_n\in\mathcal{P}$ and $\kappa_n\geq1\ a.e.$ on $\vartheta_n$,
therefore,  we get the following relation:
\begin{eqnarray*}
E[cap\, (\Theta_p )]\leq\sum_{n=p}^{+\infty}E[cap\, (\vartheta_n )]\leq\sum_{n=p}^{+\infty}E\left\|\kappa_n\right\|^2_{\mathcal{K}}\leq
2C\sum_{n=p}^{+\infty}\frac{1}{\epsilon_n^2}E\int_0^T\left(\left\|f^{0,n+1}_t-f^{0,n}_t\right\|^*_{\#}\right)^2dt,
\end{eqnarray*}
where the last inequality comes from (\ref{estimw}).\\
By extracting a subsequence, we can consider that
\begin{eqnarray*}E\int_0^T\left(\left\|f^{0,n+1}_t-f^{0,n}_t\right\|^*_{\#}\right)^2dt\leq\frac{1}{2^n} \end{eqnarray*}
and  taking $\epsilon_n=\frac{1}{n^2}$ to get
\begin{eqnarray*}E[cap\; (\Theta_p)]\leq\sum_{n=p}^{+\infty}\frac{2Cn^4}{2^n}. \end{eqnarray*}
Therefore\begin{eqnarray*}\lim_{p\rightarrow+\infty}E[cap\; (\Theta_p)]=0.\end{eqnarray*}
Finally, for almost all $\omega\in\Omega$, $w^n(\omega)$ is continuous in $(t,x)$ on $(\Theta_p(\omega))^c$ and $(w^n(\omega))$ converges uniformly to $w(\omega)$ on $(\Theta_p(\omega))^c$ for all $p$, hence, $w(\omega)$ is continuous in $(t,x)$ on $(\Theta_p(\omega))^c$, then from the definition of quasi-continuity, we know that $w(\omega)$ admits a quasi-continuous version since $cap(\Theta_p)$ tends to 0 almost surely as $p$ tends to $+\infty$.
\end{proof}

\begin{center}
\begin{minipage}[t]{7cm}
Laurent DENIS \\
Laboratoire d'Analyse et Probabilit\'es\\
 Universit{\'e} d'Evry Val
d'Essonne\\
23 Boulevard de France\\
 F-91037 Evry Cedex, FRANCE\\
 e-mail: ldenis{\char'100}univ-evry.fr
\end{minipage}
\hfill
\begin{minipage}[t]{7cm}
 Anis MATOUSSI \\
 LUNAM Université, Université du Maine\\
 Fédération de Recherche 2962 du CNRS\\
 Mathématiques des Pays de Loire \\
Laboratoire Manceau de Mathématiques\\
 Avenue Olivier Messiaen\\ F-72085 Le Mans Cedex 9, France \\
email : anis.matoussi@univ-lemans.fr\\
and \\
CMAP,  Ecole Polytechnique, Palaiseau
\end{minipage}
\hfill \vspace*{0.8cm}
\begin{minipage}[t]{12cm}
Jing ZHANG \\
Laboratoire d'Analyse et Probabilit\'es\\
 Universit{\'e} d'Evry Val
d'Essonne\\
23 Boulevard de France\\
 F-91037 Evry Cedex, FRANCE\\
Email: jing.zhang.etu@gmail.com

\end{minipage}
\end{center}


\begin{thebibliography}{99}
\addcontentsline{toc}{chapter}{Bibliographie}


\bibitem{Aronson3}  Aronson, D.G.: Non-negative solutions
of linear parabolic equations. \textit{Annali della Scuola Normale Superiore
di Pisa}, Classe di Scienze 3, tome 22 (4), pp. 607-694  (1968).

\bibitem{AronsonSerrin}  Aronson D.G. and Serrin J. : Local behavior of solutions of quasi-linear parabolic equations.
\textit{Archive for Rational Mechanics and Analysis},  25,  81-122 (1967).

\bibitem{BCEF}  Bally V.,  Caballero E., El-Karoui N. and
Fernandez, B. : Reflected BSDE's PDE's and Variational Inequalities.
\textit{preprint INRIA report} (2004).


%
%
%
\bibitem{DenisStoica}  Denis L. and Sto\"{\i}ca L.:  A general analytical
result for non-linear s.p.d.e.'s and applications.Electronic
\textit{Journal of Probability}, 9, p. 674-709 (2004).

\bibitem{DMS05}  Denis L., Matoussi A. and Sto\"{\i}ca L.:
$L^p$ estimates for the uniform norm of solutions of quasilinear
SPDE's. \textit{Probability Theory Related Fields}, 133,
437-463 (2005).

\bibitem{DMS08}  Denis L., Matoussi A. and Sto\"{\i}ca L.:
Maximum principle for parabolic SPDE's: first approach. \textit{Stohcastic
Partial Differential Equations and Applications VIII, Levico, Jan.}
6-12 (2008).

\bibitem{DMS09}  Denis L., Matoussi A. and Sto\"{\i}ca L.:
Maximum Principle and Comparison Theorem for Quasi-linear Stochastic
PDE's. \textit{Electronic Journal of Probability}, 14, p.
500-530 (2009).

\bibitem{DM11} Denis L. and Matoussi A.: Maximum principle for quasilinear SPDE's on a bounded domain without regularity assumptions.
To appear in \textit{Stoch. Proc. and their applications}, (2013).

\bibitem{DMZ12} Denis L., Matoussi A. and Zhang J.: The Obstacle Problem for Quasilinear Stochastic PDEs: Analytical approach. 
To appear in \textit{Annals of Probability}, (2013).

\bibitem{DonatiPardoux} Donati-Martin C. and Pardoux
E.: White noise driven SPDEs with reflection. \textit{Probability
Theory and Related Fields}, 95, 1-24 (1993).

\bibitem{EKPPQ}  El Karoui N., Kapoudjian C., Pardoux E., Peng S., and Quenez M.C.: Reflected Solutions of
Backward SDE and Related Obstacle Problems for PDEs. \textit{The
Annals of Probability}, 25 (2), 702-737 (1997).



\bibitem{Klimsiak} Klimsiak T.: Reflected BSDEs and
obstacle problem for semilinear PDEs in divergence form.
\textit{Stochastic Processes and their Applications}, 122
(1), 134-169 (2012).

\bibitem{Krylov}   Krylov, N. V. : An analytic approach to SPDEs. \textit{Six
Perspectives}, \textit{AMS Mathematical surveys an Monographs}, 64, 185-242 (1999).


\bibitem{LionsMagenes} Lions J.L. and Magenes
E.: Probl\`emes aux limites non homog\`enes et applications.
\textit{Dunod, Paris} (1968).
%

\bibitem{MatoussiStoica} Matoussi A. and Sto\"{\i}ca
L.: The Obstacle Problem for Quasilinear Stochastic PDE's.
\textit{Annals of Probability}, 38, 3, 1143-1179 (2010).

\bibitem{MignotPuel} Mignot F. and Puel J.P. :  In\'equations d'\'evolution paraboliques avec convexes d\'ependant
du temps. Applications aux in\'equations quasi-variationnelles
d'\'evolution. \textit{Arch. for Rat. Mech. and Ana.},  64,
No.1,  59-91 (1977).

\bibitem{NualartPardoux} Nualart D. and Pardoux
E.: White noise driven quasilinear SPDEs with reflection.
\textit{Probability Theory and Related Fields}, 93, 77-89
(1992).

\bibitem{Pierre} Pierre M.: Probl\`emes d'Evolution avec Contraintes Unilaterales et
Potentiels Parabolique. \textit{Comm. in Partial Differential
Equations}, 4(10), 1149-1197 (1979).

\bibitem{PIERRE} Pierre M. : Repr\'esentant Pr\'ecis d'Un Potentiel Parabolique. \textit{S\'eminaire
de Th\'eorie du Potentiel}, Paris, No.5, Lecture Notes in
Math. 814, 186-228 (1980).


\bibitem{RevuzYor}  Revuz, D.  and  Yor, M. : Continuous Martingales and Brownian Motion.  \textit{Springer, third edition}, (1999).
%
%
%
%

\bibitem{SW} Sanz M. , Vuillermot P. : Equivalence and Hölder Sobolev regularity of solutions for a class of non-autonomous stochastic partial differential equations. \textit{Ann. I. H. Poincaré  PR}, 39 (4), 703-742, (2003).


\bibitem{XuZhang} Xu T.G. and Zhang
T.S.: White noise driven SPDEs with reflection: Existence,
uniqueness and large deviation principles
\textit{Stochatic processes and their applications}, 119, 3453-3470 (2009).\\\\\\
\end{thebibliography}
\end{document}